\documentclass[oneside,reqno,,final,11pt]{amsart}
\usepackage{a4wide}
\usepackage{epsfig}
\usepackage[utf8x]{inputenc}
\usepackage[foot]{amsaddr}
\usepackage{amsmath,amssymb,bbm}
\usepackage{subfigure}
\usepackage[notref,notcite,color]{showkeys}
\usepackage{enumitem,color}
\usepackage{textcase}

\usepackage[colorlinks,bookmarks,linkcolor=black,citecolor=black]{hyperref}

\let\oldmarginpar\marginpar
\renewcommand\marginpar[1]{\-\oldmarginpar[\raggedleft\footnotesize #1]%
{\raggedright\footnotesize #1}}

\newtheorem{theorem}                   {Theorem}
\newtheorem{subtheorem}{Theorem}[theorem]
\newtheorem{lemma}           [theorem] {Lemma}

\newtheorem{observation}     [theorem] {Observation} 
\newtheorem{definition}      [theorem] {Definition}

\newtheorem{conjecture}      [theorem] {Conjecture}

\theoremstyle{remark} 
 
\newtheorem{claim}           [subtheorem] {Claim} 

\def\endofClaim{\hfill\scalebox{.6}{$\Box$}}
\newcommand{\oldqed}{}
\newenvironment{claimproof}[1][Proof]{
  \renewcommand{\oldqed}{\qedsymbol}
  \renewcommand{\qedsymbol}{\endofClaim}
  \begin{proof}[#1]
}{
  \end{proof}
  \renewcommand{\qedsymbol}{\oldqed}
}

\let\subset\subseteq
\let\epsilon\varepsilon
\let\eps\varepsilon
\let\rho\varrho

\def\cU{\mathcal U}
\def\cW{\mathcal W}

\def\cC{\mathcal C}
\def\cF{\mathcal F}
\def\cP{\mathcal P}
\def\cT{\mathcal T}

\def\RR{\mathbb{R}}

\setlist{itemsep=1pt,parsep=0pt,topsep=3pt,partopsep=0pt,itemsep=2pt}  
\setlist{leftmargin=*,labelindent=.5\parindent}

\def\abc{\rm (\alph{*})}

\def\dcup{\dot\cup}

\def\subsc#1{\textsc{\MakeTextLowercase{#1}}} 

\def\qquand{\qquad\text{and}\qquad}

\def\st{ :\quad } 

\def\codeg{\mathrm{codeg}}

\def\neighbor{\mathrm{N}}
\def\component{\mathrm{Comp}}

\def\FN{\mathrm{FN}}
\def\YN{\mathrm{YN}}
\def\XN{\mathrm{XN}}
\def\RN{\mathrm{RN}}
\def\VC{\mathrm{VC}}
\def\EC{\mathrm{EC}}
\def\dist{\mathrm{dist}}

\def\load{\mathrm{load}}

\newcommand{\By}[2]{\overset{\mbox{\tiny{#1}}}{#2}}
\newcommand{\ByRef}[2]{   \By{\eqref{#1}}{#2} }

\newcommand{\leBy}[1]{    \By{#1}{\le} }

\newcommand{\eqByRef}[1]{ \ByRef{#1}{=} }

\newcommand{\leByRef}[1]{ \ByRef{#1}{\le} }
\newcommand{\geByRef}[1]{ \ByRef{#1}{\ge} }
\newcommand{\DEF}[1]{\emph{#1}\marginpar{#1}}

\newcommand{\Exp}{\mathbb{E}}

\newcommand{\Prob}{\mathbb{P}}

\def\primary{\mathsf{prim}}
\def\secondary{\mathsf{sec}}
\def\impgr{\mathsf{IG}} 

\def\alphaA{\alpha_\textsf{A}}
\def\alphaB{\alpha_\textsf{B}}
\def\alphaC{\alpha_\textsf{C}}
\def\alphaD{\alpha_\textsf{D}}
\def\alphaE{\alpha_\textsf{E}}

\def\bad{\mathrm{BAD}}
\def\badness{\mathrm{bad}}

\title{An approximate version of the Tree Packing Conjecture}
\thanks{Support by the Institut Mittag-Leffler
(Djursholm, Sweden) is gratefully acknowledged.}
\author{Julia B\"ottcher$^\dagger$}
\address{$^\dagger$ Department of Mathematics, London School of Economics, Houghton Street, London, WC2A~2AE, UK.}
\author{Jan Hladk\'y$^\ddagger$}
\address{$^\ddagger$ Mathematics Institute, Czech Academy of Sciences, \v Zitn\'a 25, Praha. Mathematics Institute is supported by RVO:67985840. The work was done while the author was an EPSRC Research Fellow at DIMAP and Mathematics Institute at the University of
Warwick.}
\author{Diana Piguet$^*$}
\address{$^*$ European Centre of Excellence
NTIS -- New Technologies for Information Society
Faculty of Applied Sciences,
University of West Bohemia,
Univerzitn\'i\ 22, 306 14 Plze\v n,
Czech Republic. \textit{The research leading to these results has received funding from the European Union Seventh
Framework Programme (FP7/2007-2013) under grant agreement no.~PIEF-GA-2009-253925.}}
\author{Anusch Taraz$^\star$}
\address{$^\star$ Institut f\"ur Mathematik, Technische Universit\"at Hamburg-Harburg,
Schwarzenbergstrasse 95, Geb\"aude E, 21073 Hamburg, Germany.
\textit{The author was supported in part by DFG grant TA 319/2-2.}}
\email{j.boettcher@lse.ac.uk, honzahladky@gmail.com, piguet@ntis.zcu.cz, taraz@tuhh.de}
\begin{document}
\begin{abstract}
We prove that for any pair of constants $\eps>0$ and $\Delta$ and for $n$ sufficiently large, every family of trees of orders at most~$n$, maximum degrees at most $\Delta$, and with at most~$\binom{n}{2}$ edges in total packs into $K_{(1+\eps)n}$. This implies asymptotic versions of the Tree Packing Conjecture of Gy\'arf\'as from~1976 and a tree packing conjecture of Ringel from~1963 for trees with bounded maximum degree. A novel random tree embedding process combined with the nibble method forms the core of the proof.
\end{abstract}
\maketitle

\section{Introduction}\label{sec:intro}
Graph packing is a concept that generalises the notion of graph embedding to
finding several subgraphs in a host graph instead of just one.
A family of graphs $\mathcal{H}=(H_1,\dots,H_k)$ is said to
\DEF{pack} into a graph~$G$ if there exist pairwise
edge-disjoint copies of $H_1,\dots,H_k$ in $G$, where we allow $H_i=H_j$ for $i\neq j$.
Many classical problems in Graph Theory can be
stated as packing problems. 
For example, Mantel's Theorem can be formulated by saying that 
if $G$ is an $n$-vertex graph with less than
$\binom{n}{2}-\frac{n^2}{4}$ edges, then the family $(K_3,G)$
packs into $K_n$.

Among the best known packing problems, let us for example 
mention a conjecture of Bollob\'as, Catlin, and Eldridge~\cite{BollEld:PackingConj,Catlin:PackingConj} that
any two $n$-vertex graphs $H_1,H_2$ of maximum degree $\Delta(H_1)$ and $\Delta(H_2)$, respectively, and satisfying
$(\Delta(H_1)+1)(\Delta(H_2)+1)\le n+1$
pack into $K_n$. The asymptotic solution of this conjecture was reported by G\'abor Kun around~2006.

Another beautiful packing conjecture was posed by Gy\'arf\'as (see \cite{GyaLeh}) in~1976 and concerns trees. This conjecture is
referred to as the Tree Packing Conjecture.
\begin{conjecture}\label{conj:tree-packing}
Any family $(T_1,T_2,\ldots,T_n)$ of trees, $j\in
[n]$ of order $v(T_j)=j$, packs into~$K_n$. 
\end{conjecture}

A related conjecture of
Ringel~\cite{Ringel:TreePacking} dating back to~1963
deals with packing many copies of the same tree.
\begin{conjecture}\label{conj:Ringel}
Any $2n+1$ identical copies of any tree of order $n+1$
pack into $K_{2n+1}$.
\end{conjecture}

Note that both conjectures are best possible in the sense that they deal
with \DEF{perfect packings}, i.e.~the total number of edges packed equals
the number of edges in the host graph.  Moreover, the fact that two
spanning stars do not pack into the complete graph shows that further
requirements than this necessary condition are needed.

A slightly outdated survey on packings of trees is by
Hobbs~\cite{Hobbs81:PackingSurvey}. Here, we recall only
the most important results concerning the two
conjectures above. 

A packing of many of the small trees from Conjecture~\ref{conj:tree-packing}
was obtained by Bollob\'as~\cite{Bollobas83:Packing}, who showed
that any family of trees $T_1,\ldots, T_s$ with $v(T_i)=i$ and $s<
n/\sqrt{2}$ can be packed into $K_n$. He also observed that the validity of
a famous conjecture of Erd\H{o}s and S\'os would imply that one can
improve the bound to $s<\frac12\sqrt{3}n$. The Erd\H{o}s-S\'os Conjecture states that any graph of average degree greater than $k-1$ contains any tree of order at most $k+1$ as a subgraph. The solution of this conjecture for large trees was announced by Ajtai,
Koml\'os, Simonovits, and Szemer\'edi in the early~1990s.
In a similar direction,
Yuster~\cite{Yuster97:Packing} proved that any sequence of
trees $T_1,\ldots,T_s$, $s<\sqrt{5/8}n$ can be packed into
$K_{n-1,n/2}$.
This improves upon a result of Caro and
Roditty~\cite{CaRo90:PackingTrees} and is related to a conjecture of Hobbs, Bourgeois and Kasiraj~\cite{HBK87} 
(see Conjecture~\ref{conj:tree-packingBip} in Section~\ref{sec:Concluding}).
Moreover, a result of Caro and Yuster~\cite{CaYu97:Packing} implies that one can pack
perfectly a family of trees into a complete
graph $K_n$, provided that the trees are very small compared to $n$.

Packing the large trees of Conjecture~\ref{conj:tree-packing}
is a much more challenging task.  Balogh and Palmer~\cite{BaPa:TreePacking}
proved that any family of trees $T_n$, $T_{n-1}$, \ldots,
$T_{n-\frac1{10}n^{1/4}}$, $v(T_i)=i$ packs into~$K_{n+1}$.

Surprisingly few results are known for special classes of tree families.
It was proved already in~\cite{GyaLeh} that
Conjecture~\ref{conj:tree-packing} holds when all the trees are stars and paths. 
Dobson~\cite{Dob02:Packing} and  Hobbs, Bourgeois, and Kasiraj~\cite{HBK87} consider packings of trees with small diameter.
Moreover, Fishburn~\cite{Fish83:MatrixPacking} proved that it is
at least possible to adequately match up the degrees of the trees
$T_1,\dots,T_{n}$ appearing in Conjecture~\ref{conj:tree-packing}: If we
add $n-i$ isolated vertices to the tree~$T_i$ and let
$d_{i,1},\dots,d_{i,n}$ denote the degree sequence of the resulting
forest, then there are permutations $\pi_1,\dots,\pi_n$ such that $\sum_i
d_{i,\pi_i(j)}=n-1$ for all $j\in[n]$.
\bigskip

Our main result, Theorem~\ref{thm:tree-packing},
deals with almost perfect packings of
bounded-degree trees into a complete graph. It implies an
asymptotic solution of Conjecture~\ref{conj:tree-packing} 
and Conjecture~\ref{conj:Ringel} for trees of bounded maximum
degree.
\begin{theorem}\label{thm:tree-packing}
  For any $\eps>0$ and any $\Delta\in \mathbb N$ there is an $n_0\in
  \mathbb N$ such that for any $n\ge n_0$ the following holds. Any family
  of trees $\mathcal T=(T_i)_{i\in[k]}$ such that~$T_i$ has maximum degree
  at most~$\Delta$ and order at most~$n$ for each $i\in[k]$, and
  $\sum_{i\in[k]} e(T_i)\le{\binom{n}{2}}$ packs into $K_{(1+\eps)n}$.
\end{theorem}

We emphasise that, unlike Conjectures~\ref{conj:tree-packing} and~\ref{conj:Ringel}, this theorem only requires the trees to satisfy
the obvious upper bound on the total number of edges.

\section{Outline of the proof}\label{sec:outline}
A very natural approach to pack the trees $T_1,\dots,T_k$ into
$K_{(1+\eps)n}$ is to use a \emph{random embedding process}:
\begin{itemize}
\item Start with $G=K_{(1+\eps)n}$.  Successively build a packing of the
  trees, edge by edge, starting with an arbitrary edge in an arbitrary tree
  and then following the structure of the trees (it is not important which
  order exactly we choose, but one example would be to use a breadth-first
  search order; it also should not matter here whether we embed tree by
  tree, or first embed a few edges of one tree, then a few edges of another
  tree, and so on, and then return to the first tree).
\item In one step of this procedure, when we want to embed an edge $xy$ of
  some tree~$T_i$, with~$x$ already embedded to $h(x)$, choose a random
  neighbour~$v\in V(G)$ of $h(x)$ which is not contained in
  the set~$U_i\subset V(G)$ of $T_i$-images so far, and embed~$y$ to $h(y)=v$.
\item After embedding $xy$, remove the edge $uv$ from~$G$ and add~$v$
  to~$U_i$.
\end{itemize}
Clearly, this process produces a proper packing unless we get stuck, that
is, unless the set $\neighbor_G\big(h(x)\big)\setminus U_i$ gets empty. But if, during the evolution, the host graph~$G$ always remains
sufficiently quasirandom, then with high probability
$\neighbor_G\big(h(x)\big)\setminus U_i$ should not get empty (because
$e(K_{(1+\eps)n})-\sum_{i\in[k]} e(T_i)\ge\eps n^2$ implies that~$G$ has
positive density throughout).

We believe that the host graph does indeed remain quasirandom in this
process. Unfortunately, however, graph processes like this are extremely
difficult to analyse because of their dynamically evolving environment in
each step. A prominent example illustrating the occurrent complexity is
that of the random triangle-free graph process: it took more than a decade
after the introduction of this process until Bohman~\cite{Bohman:Ramsey}
gave a detailed analysis. Nonetheless, a related random construction of
triangle-free graphs was effectively analysed already much earlier by
Kim~\cite{Kim:Ramsey}. This construction was easier to handle because it
uses a \emph{nibble} approach.

The nibble method bypasses the
difficulties originating from the dynamics of random graph processes by
proceeding in constantly many rounds and updating the environment only after each round.
This method was used by R\"odl~\cite{RodlNibble85} to prove the
existence of asymptotically optimal Steiner systems (see~\cite{AS08} for
an exposition). 
Since then it has served as an important ingredient for
several breakthroughs in combinatorics. 
In the context of packing problems the nibble method
is also used
in Kun's announced result on the Bollob\'as--Catlin--Eldridge Conjecture.
In our setting the nibble method amounts to the following approach for
embedding $T_1,\dots,T_k$ into $G=K_{(1+\eps)n}$:
\begin{itemize} 
\item Pack the trees in~$r$ rounds (with~$r$ big but constant). For this purpose,
  cut each tree~$T_i$ into small equally sized forests~$F_{i}^j$ with
  $j\in[r]$ and use in each round exactly one forest of each tree.
\item In round~$j$, for each~$i$ construct a \emph{random homomorphism}
  from the forest~$F_{i}^j$ to~$G$ as follows. First, randomly embed some
  forest vertex~$x$, 
then choose a neighbour~$v$ uniformly at random in 
$\neighbor_G(h(x))\setminus U_i$, where the \emph{forbidden set}
  $U_i\subset V(G)$ are vertices used by~$T_i$ in previous rounds. Then
  continue with the next vertex in~$F_i^j$, following again the structure
  of~$T_i$.
\item After round~$j$, delete all the edges from~$G$ to which some forest edges
  were mapped in this round and add to $U_i$ all images of vertices of $F_{i}^j$.
\end{itemize}
In other words, the difference between this approach and the random process
described above is that the host graph~$G$ and the sets~$U_i$ are not
updated after the embedding of each single vertex, but only at the end of
each round.

Naturally, this procedure will not produce a proper packing of the trees:
Firstly, it will create \emph{vertex collisions}, that is, two vertices of
some tree~$T_i$ are mapped to the same vertex of the host graph~$G$.
Secondly, there will be \emph{edge collisions}, that is, two edges of
different trees are mapped to the same edge. However, since all
forests~$F_i^j$ are small this will create only a small proportion of vertex and edge
collisions in each round, and the updates at the end of each round guarantee that
there are no collisions between rounds. So our hope is that vertex and edge
collisions can be \emph{corrected} at the end.

The difficulty with this construction of random homomorphisms though is
that it still leads to lots of small dependencies between embedded
vertices, which we found difficult to control. We remark that techniques
recently developed by Barber and Long~\cite{BarberLong} allow to handle
these dependencies and show that after each round the host graph is indeed
quasirandom. However, applying these techniques to our setting and
modifying them so that they also give all the additional properties that
we need (such as that there are few collisions; see
Lemma~\ref{lem:newnibble}) would require substantial additional work and
probably lead to a significantly longer proof.

Our approach (which was developed before the techniques of Barber and Long)
is different. We instead use the following construction of random
homomorphisms in round~$j$ of the nibble approach described above, which we
call \emph{limping homomorphisms}:
\begin{itemize} 
\item For each~$i$, call one of the colour classes of~$F_i^j$ the set of
  \emph{primary vertices}, and the other the set of \emph{secondary
    vertices}.
  Now first map all primary vertices randomly to vertices of $V(G)\setminus
  U_i$. Then
  map each secondary vertex randomly into the common
  $(G-U_i)$-neighbourhood of the
  images of its forest neighbours -- unless this common neighbourhood is
  smaller than expected, in which case we simply \emph{skip} this secondary
  vertex.
\end{itemize}
Observe that, if our host graph is quasirandom (and the forest has bounded
degree), then most common neighbourhoods are big and hence few vertices
will get skipped. Of course in this random construction we still have
dependencies. But since these occur only between vertices with distance at
most~$2$ in the trees, we now can control them and prove that the host
graph is quasirandom after each round and that we get few collisions.

It remains to correct the vertex collisions and edge collisions (and take
care of skipped vertices and connections between the different forests of
each tree). Before starting the
described embedding rounds we put aside $\eps n/2$ reserve vertices of
$K_{(1+\eps)n}$. Our random homomorphisms
(constructed on the remaining vertices) also guarantee that the collisions are
sufficiently well distributed over the host graph so that a simple greedy
strategy can be used to relocate vertices in collisions to the reserved
vertices, thus obtaining a proper packing of $T_1,\dots,T_k$.

\medskip
The organisation of the proof is given in Table~\ref{tab:outline}.
\begin{table}
\centering
\begin{tabular}{ccc}
\hline
\multicolumn{3}{|c|}{Theorem~\ref{thm:tree-packing}; proof in Section~\ref{sec:mainproof}}\\
\hline
\multicolumn{1}{c}{$\Uparrow$}& &\multicolumn{1}{c}{$\Uparrow$}\\
\cline{1-1}\cline{3-3}
\multicolumn{1}{|c|}{Packing with a small number of collisions}& &\multicolumn{1}{|c|}{Correcting collisions}\\
\multicolumn{1}{|c|}{Lemma~\ref{lem:ExistsAlmostPacking}; proof in Section~\ref{sec:ExistsAlmostPacking}}& &\multicolumn{1}{|c|}{Lemma~\ref{lem:almostpack}; proof in Section~\ref{sec:almostpack}}\\
\cline{1-1}\cline{3-3}
\multicolumn{1}{c}{$\Uparrow$}& & \\
\cline{1-1}
\multicolumn{1}{|c|}{One round: Lemma~\ref{lem:newnibble}; proof in Section~\ref{sec:proofofnibble}.}& &\\
\multicolumn{1}{|l|}{The proof builds on properties of limping}& &\\
\multicolumn{1}{|l|}{homomorphisms derived in Section~\ref{sec:limpingWalk}.}& &\\
\cline{1-1}\\
\end{tabular}
\caption{Outline of the proof of Theorem~\ref{thm:tree-packing}.}
\label{tab:outline}
\end{table}

\section{Proof of the main theorem (Theorem~\ref{thm:tree-packing})}\label{sec:mainproof}

Theorem~\ref{thm:tree-packing} assumes little on the orders of trees~$\mathcal T$ to be packed. However, as we show as a first step of the proof of Theorem~\ref{thm:tree-packing}, there is a simple transformation of an arbitrary such family into a family of trees whose orders are (with possibly one exception) more than $n/2$. The definition of an $(n,\Delta)$-tree family below formalises this. The fact that the subsequent family is a family of trees of linear orders is crucial for our proof.
\begin{definition}
 A family of trees $\mathcal{T}$ is called \DEF{$(n,\Delta)$-tree family}, if
  all trees in $\mathcal T$ have order at most $n$, maximum degree $\Delta$
  and the total number of edges is at most $\binom{n}{2}$. Further all but
  at most one tree from $\mathcal T$ have order more than $n/2$. Observe
  that the upper bound on the total number of edges and the lower bound on
  the number of vertices imply that such a family must contain less than
  $2n$ trees.
\end{definition}

Indeed, it is easy to show that we can transform any family $\mathcal{T}$ satisfying the requirements of Theorem~\ref{thm:tree-packing} into an $(n,\Delta)$-tree family (see below).  

Our next step will be to relax the requirements of a packing in the sense that we allow an exceptional set $R$ of vertices \emph{not} to be embedded. At the same time, we control both the size of $R$ as well as the number of neighbours of $R$ that get embedded into the same vertex. 

\begin{definition}[Almost packing]\label{def:almostpack}
Let $\mathcal F=\{F_i\}_{i\in[k]}$ be a family of graphs. 
For a graph~$G$, a family of sets $\{R_i\}_{i\in[k]}$ with $R_i\subset V(F_i)$ and a family of maps
$\{h_i\}_{i\in[k]}$ with $h_i:V(F_i)\setminus R_i\rightarrow V(G)$ we say that
$\{h_i,R_i\}_{i\in[k]}$ is an \DEF{$\ell$-almost packing} of $\mathcal F$ into~$G$ if
\begin{enumerate}[label=\abc]
\item $\{h_i\}_{i\in[k]}$ is a packing of the family $\{F_i-R_i\}_{i\in[k]}$ into the graph $G$,
\item\label{def:almostpack:R} we have $|R_i|\le \ell$ for each $i\in [k]$, and
\item\label{def:almostpack:RN} for each $v\in V(G)$, \  
$\sum_{j\in[k]} |\big\{x\in h^{-1}_{j}(v)\colon\, \exists xy\in E(F_{j})
\text{ such that $y\in R_j$}\big\}| \le \ell$.
\end{enumerate}

We say that $\mathcal F$ \DEF{$\ell$-almost packs} into a graph $G$ if there exist $\{R_i\}_{i\in[k]}$ and $\{h_i\}_{i\in[k]}$ such that $\{h_i,R_i\}_{i\in[k]}$ is an $\ell$-almost packing of $\mathcal F$ into~$G$.
\end{definition}

Using this concept, the next two lemmas state that we can always find an almost packing,
and that an almost packing can always be turned into a packing.

%

\begin{lemma}[Almost packing lemma]\label{lem:ExistsAlmostPacking}
For any $\eps>0$ and any $\Delta\in \mathbb N$ there is an $n_0\in \mathbb N$
such that for any $n\ge n_0$ the following holds. Any $(n,\Delta)$-tree family $(\epsilon n)$-almost packs into~$K_{(1+\eps)n}$.
\end{lemma}

\begin{lemma}[Correction lemma]\label{lem:almostpack}
Let $\epsilon>0$ be arbitrary, and let $\mathcal T$ be a family of trees of maximum degrees at most $\Delta$. Suppose that $|\mathcal T|\le 2m$, and that 
$\mathcal T$ has an $(\frac{\epsilon^2 m}{64\Delta^2})$-almost packing into~$K_m$. 
Then $\mathcal T$ packs into $K_{(1+\epsilon) m}$.
\end{lemma}

Lemmas~\ref{lem:ExistsAlmostPacking} and~\ref{lem:almostpack} are proven in Section~\ref{sec:ExistsAlmostPacking} and Section~\ref{sec:almostpack}, respectively.
Based on these two lemmas, it is now an easy task to prove our main theorem.
We remark that here and in the rest of the paper, we shall often use subscripts on constants to clarify which theorem/lemma they originate from: For example $\eps_{\subsc{L\ref{lem:almostpack}}}$ refers to the constant $\eps$ from Lemma~\ref{lem:almostpack}.

\begin{proof}[Proof of Theorem~\ref{thm:tree-packing}]

Let $\eps>0$ and $\Delta\in \mathbb N$ be given.
We define $\eps_{\subsc{L\ref{lem:ExistsAlmostPacking}}}=\eps^2/(256\Delta^2)$ and apply Lemma~\ref{lem:ExistsAlmostPacking} 
with parameters $\eps_{\subsc{L\ref{lem:ExistsAlmostPacking}}}$ and $\Delta$ to obtain $n_0$. 

Now we consider a family $\mathcal T$ of trees satisfying the requirements of the theorem. If $\mathcal T$ contains two trees $F_1$ and $F_2$ of orders at most $n/2$, then we can replace them by a single tree of order $v(F_1)+v(F_2)-1$ that is obtained by identifying
a leaf of $F_1$ and a leaf of $F_2$. Repeating this step, we arrive at a situation where all but at most one tree in $\mathcal T$ have order more than $n/2$. This procedure does not change the maximum degree of the 
trees nor their total number of edges. Hence we have obtained an $(n,\Delta)$-tree family $\mathcal T'$. Observe that now it suffices to pack $\mathcal T'$ into $K_{(1+\eps)n}$.
Feeding the family $\mathcal T'$ to Lemma~\ref{lem:ExistsAlmostPacking}, 
we obtain an $(\eps_{\subsc{L\ref{lem:ExistsAlmostPacking}}} n)$-almost packing of $\mathcal T'$ 
into $K_{(1+\eps_{\subsc{L\ref{lem:ExistsAlmostPacking}}})n}$. 

Now we set $m=(1+\frac{\eps}{4})n \ge (1+\eps_{\subsc{L\ref{lem:ExistsAlmostPacking}}})n$ and 
$\eps_{\subsc{L\ref{lem:almostpack}}}=\eps/2$. Since
\[
\eps_{\subsc{L\ref{lem:ExistsAlmostPacking}}}n = \frac{\eps^2 n}{256\Delta^2} \le \frac{\eps^2 m}{256\Delta^2}
= \frac{(\eps_{\subsc{L\ref{lem:almostpack}}})^2 m}{64\Delta^2},
\]
the family $\mathcal T'$ also has an $\frac{(\eps_{\subsc{L\ref{lem:almostpack}}})^2 m}{64\Delta^2}$-almost packing 
into $K_m$. As the number of trees in $\mathcal T'$ is bounded by $2n\le 2m$, we can apply Lemma~\ref{lem:almostpack} with parameter $\eps_{\subsc{L\ref{lem:almostpack}}}$ 
and obtain a packing of $\mathcal T'$ into $K_{(1+\eps_{\subsc{L\ref{lem:almostpack}}})m}$. 
Since $(1+\eps_{\subsc{L\ref{lem:almostpack}}})m = (1+\frac{\eps}{2})(1+\frac{\eps}{4})n \le (1+\eps)n$, this completes the proof.  
\end{proof}

\section{Notation and preliminaries}

\subsection{Basic notation}
Let $G=(V,E)$ be a graph and $V'\subset V$ and $E'\subset E$. We use the minus
symbol to denote both the removal of vertices and edges from a graph, i.e.,
$G-V'=(V\setminus V',E\cap \binom{V'}{2})$, $G-E'=(V,E\setminus E')$. For vertex sets $U,W\subset V$ we let\marginpar{$e(U), e(U,W)$} $e(U)$ denote the number of edges with both endvertices in~$U$ and let
$e(U,W)=|\{(u,w)\in U\times W: uw\in E\}|$. Here, the edges with both
endvertices in $U\cap W$ are counted twice. The \DEF{common neighbourhood} of vertices $v_1,\ldots,v_k$ in the graph $G$ is defined by $\neighbor_G(v_1,\ldots,v_k)=\{u\in V: uv_1,uv_2,\ldots,uv_k\in E\}$. The 
\DEF{codegree} of $v_1,\ldots,v_k$ is then \marginpar{$\codeg_G(v_1,\ldots,v_k)$}$\codeg_G(v_1,\ldots,v_k)=|\neighbor_G(v_1,\ldots,v_k)|$. In the special case $k=1$, this quantity is called the \DEF{degree} of $v_1$, \marginpar{$\deg_G(v)$}$\deg_G(v_1)=\codeg_G(v_1)$. We drop the subscript when the graph $G$ is understood from the context. The \DEF{density} of~$G$ is defined as $|E|/\binom{|V|}{2}$.

Denote by \marginpar{$\dist(x,y)$}\marginpar{$\dist(U,W)$}$\dist(x,y)$ the length of a shortest path between $x$ and $y$. Here, the distance between vertices lying in different components is defined to be $+\infty$. For two sets $U,W$ of vertices of the same graph we write $\dist(U,W)=\min_{u\in U, w\in W}\dist(u,w)$. In particular, we will use this notation when $U$ and $W$ are edges (i.e., vertex sets of size two).

By a $d$-th \DEF{power} of a graph $G=(V,E)$ we mean its distance-power, that is, a loopless graph, denoted $G^d$, on the vertex set $V$ where two vertices $u$ and $v$ are adjacent if and only if $\dist_G(u,v)\le d$. We refer to the case $d=2$ as \DEF{square}.

Finally, the set of components of~$G$ is denoted by\marginpar{$\component(G)$} $\component(G)$.

\medskip
Generally, we shall use letters  $x$, $y$, and $z$ to denote vertices in trees and forests that we pack. Letters $u$, $v$, and $w$ will be used to denote the vertices in the host graph into which we pack. When we write $a=b\pm c$, we mean that~$a$ has its value in the interval $[b-c,b+c]$\marginpar{$\pm$}. Analogously, by $a\neq b\pm c$ we mean that~$a$ has its value outside the interval $[b-c,b+c]$.

\subsection{Quasirandomness}

Here, we recall the concept of quasirandom graphs, which goes back to
Thomason~\cite{Thomason87:Pseudorandom}, and Chung, Graham, and
Wilson~\cite{ChGraWi89:Quasirandom}.

\begin{definition}[Quasirandom graph]\label{def:subsetQuasi}
  We say that a graph $G$ of order $n$ is \DEF{$\alpha$-quasirandom of
    density $d$} if for every $B\subset V(G)$ we have $e(B)=d\binom{|B|}{2}\pm
  \alpha n^2$ edges.
\end{definition}

Since $e(A,B)=e(A\cup B)+e(A\cap B)-e(A\setminus B)-e(B\setminus A)$, this
definition immediately implies that in a quasirandom graph we also have control
over the number of edges between two
vertex sets.

\begin{observation}\label{obs:tunel}
  In an $\alpha$-quasirandom graph~$G$ on $n$ vertices, for each pair of
  sets $A,B\subset V(G)$ we have $e(A,B)=d|A||B|\pm 4\alpha n^2\pm n$.
\end{observation}

Our next easy lemma asserts that induced subgraphs of quasirandom
graphs inherit quasirandomness and density.
\begin{lemma}
\label{prop:SubgraphOfQuasiRandomGraph}
If~$G$ is $\alpha$-quasirandom of density~$d$ and order at most $\frac32 n$, and a set $V'\subset V(G)$ has size $|V'|\ge\eps n$, then $G[V']$ is a $(3\alpha/\eps^2)$-quasirandom graph of density $d\pm 3\alpha/\eps^2$.
\end{lemma}
\begin{proof}
For any $B\subset V'$ we have
  \begin{equation*}\begin{split}
    e_{G}(B)&=d\binom{|B|}{2}\pm\alpha(\tfrac32n)^2
    =d\binom{|B|}{2}\pm\alpha\cdot (\tfrac32)^2\cdot\frac{|V'|^2}{\eps^2}
    =d\binom{|B|}{2}\pm3\frac{\alpha}{\eps^2}|V'|^2\;.
  \end{split}
  \end{equation*}
  Hence~$G[V']$ is a $(3\alpha/\eps^2)$-quasirandom graph of density $d\pm 3\alpha/\eps^2$.
\end{proof}

If~$G=(V,E)$ is a quasirandom graph with density~$d$, we
expect that in~$G$ most sets of $p$ vertices have a common neighbourhood of order roughly $d^p |V|$. 
So, we say that  
a set $\{v_1,\dots,v_p\} \subset V$ is \DEF{$\gamma$-bad}, if
\[
\left|\neighbor(v_1,\dots,v_p)\right| \neq (1\pm \gamma) d^p |V|.
\]
The next lemma states that most vertices of a quasirandom graph are
contained in few bad $p$-sets. We use the following definitions. For a
vertex $v\in V$, let \marginpar{$\badness_{\gamma,p}(v)$}
\[
\badness_{\gamma,p}(v)=\Big|\left\{B\in\tbinom{V}{p-1}:
  \text{$B\cup\{v\}$  is $\gamma$-bad}\right\}\Big|\;.
\]
(In particular, $\badness_{\gamma,1}(v)\in\{0,1\}$, depending on whether $\deg(v)=(1\pm\gamma)d|V|$, or not.) Set
\marginpar{$\bad_{\gamma,\Delta}(G)$}
\[
\bad_{\gamma,\Delta}(G)= \left\{ 
v \in V\::\: 
\badness_{\gamma,p}(v) > \gamma \tbinom{|V|}{p-1} 
\text{ for some } p\in [\Delta]
\right\}.
\]

\begin{lemma}\label{lem:bad}
  For every $\gamma>0$ and every integer $\Delta\ge 1$ there is $\alpha>0$
  such that if~$G=(V,E)$ is an $\alpha$-quasirandom graph of density $d\ge\gamma$ and
  order~$n$, then $|\bad_{\gamma,\Delta}(G)|\le\gamma n$.
\end{lemma}
\begin{proof}
  Let $\alpha\le 1/(10\Delta^2)$ be small enough so that for $\beta=\frac12\sqrt{\alpha}$ and $\gamma_1\le\dots\le\gamma_\Delta$ defined by
  \begin{equation*}
    \gamma_p = \begin{cases}
      \sqrt{\frac{10\beta}{d}} & p=1 \\
      \sqrt{4p\gamma_{p-1} + \frac{20p\beta}{d^p\gamma_{p-1}}} \quad & 1<p\le\Delta \,,
    \end{cases}
  \end{equation*}
  we have $\gamma_\Delta\le\min\{\gamma/\Delta,1/2\}$. Testing over two-element sets in Definition~\ref{def:subsetQuasi}, we get that if $n<\max\{2\Delta,\beta^{-1}\}$ then $G$ is either complete or empty. Hence we may assume that $n\ge \max\{2\Delta,\beta^{-1}\}$ in the following.

  We prove by induction on~$p$ that 
  \begin{equation}\label{eq:IMP}
  \mbox{
  at most $\gamma_p n$ vertices~$v$ of
  $G$ satisfy $\badness_{\gamma_p,p}(v)>\gamma_p\tbinom{n}{p-1}$.
  }
  \end{equation}
 Let us
  first consider the base
  case $p=1$. Let~$V^+$ be the set of
  vertices~$v$ with $\deg(v)>(1+\gamma_1)dn$. We have $e(V^+,V)>|V^+|
  (1+\gamma_1)dn$. But since~$G$ is $\alpha$-quasirandom we have by Observation~\ref{obs:tunel} that
  $e(V^+,V)\le d|V^+|n+4\alpha n^2+n\le d|V^+|n+5\beta n^2$. Putting these bounds together, we get $|V^+|<5\beta n/(d\gamma_1)$. Similarly for the
  set~$V^-$ of vertices~$v$ with $\deg(v)<(1-\gamma_1)dn$ we have
  $|V^-|<5\beta n/(d\gamma_1)$. Thus there are at most $10\beta 
  n/(d\gamma_1)=\gamma_1 n$ vertices~$v$ with
  $\badness_{\gamma_1,1}(v)=1>\gamma_1\tbinom{n}{0}$.

  Now consider $p>1$ and assume that~\eqref{eq:IMP} holds for $p-1\ge 1$. The number of $\gamma_{p-1}$-bad sets in $\binom{V}{p-1}$ is
  \begin{equation}\label{eq:bad}
  \frac{1}{p-1}\sum_{v\in V}\badness_{\gamma_{p-1},p-1}(v)\le
  \frac{1}{p-1}\left(\gamma_{p-1}n\binom{n}{p-2}+n\gamma_{p-1}\binom{n}{p-2}\right)\le 4\gamma_{p-1}\binom{n}{p-1}\;,
  \end{equation}
  where we used $n/2\le n-p+1$.  Fix an arbitrary set
  $\{v_1,\dots,v_{p-1}\}$ in $\binom{V}{p-1}$ that is not
  $\gamma_{p-1}$-bad. Hence for $W=\neighbor(v_1,\dots,v_{p-1})$ we have
  $|W|=(1\pm\gamma_{p-1})d^{p-1}n$.  Let~$V^+$ be the set of vertices~$v\in
  V\setminus\{v_1,\dots,v_{p-1}\}$
  with $|\neighbor(v)\cap W|>(1+\gamma_{p-1})d|W|$. We have
  $|V^+|(1+\gamma_{p-1})d|W|<e(V^+,W)\le d|V^+||W|+5\beta n^2$ and
  hence $|V^+|<5\beta n^2/(\gamma_{p-1}d|W|)\le 5\beta n^2/(\gamma_{p-1}d\frac12d^{p-1}n)=
  10\beta n/(d^p\gamma_{p-1})$. Similarly, for the set $V^-$ of vertices~$v$ such that
  $|\neighbor(v)\cap W|<(1-\gamma_{p-1})d|W|$ we have $|V^-|<10\beta n/(
  d^p\gamma_{p-1})$. Let~$v$ be an arbitrary vertex in $V\setminus (V^+\cup V^-\cup \{v_1,\ldots,v_{p-1}\})$. Then
  \begin{equation*}
    |\neighbor(v,v_1,\dots,v_{p-1})| = (1\pm\gamma_{p-1})d|W| = (1\pm
    \gamma_p)d^pn \,,
  \end{equation*}
  and therefore $\{v,v_1,\dots,v_{p-1}\}$ is not $\gamma_p$-bad. Hence,
  by~\eqref{eq:bad}, the number of $\gamma_p$-bad $p$-tuples is at most
  \begin{equation*}
    \Big(4\gamma_{p-1} \binom{n}{p-1}\Big) n
    + \binom{n}{p-1}\cdot 2\frac{10\beta n}{d^p\gamma_{p-1}}
    =
   \frac{\gamma_p^2}p n\binom{n}{p-1}\;.   
  \end{equation*}
   Consequently, for at most $\gamma_{p}n$ vertices $v\in V$, we have 
  $\badness_{\gamma_p,p}(v)>\gamma_p \binom{n}{p-1}$. This gives~\eqref{eq:IMP}.

  The bound $|\bad_{\gamma,\Delta}(G)|\le\gamma n$ follows by summing~\eqref{eq:IMP} over $p=1,\ldots,\Delta$. 
\end{proof}

As our next lemma shows, this implies that we need to delete only few
vertices from a quasirandom graph to obtain a graph~$G$ in which
$\bad_{\gamma,\Delta}(G)=\emptyset$. 

\begin{definition}[Superquasirandom graph]
  We say that a graph $G$ is \DEF{$(\gamma,\Delta)$-super\-quasi\-ran\-dom} if
  we have $\bad_{\gamma,\Delta}(G)=\emptyset$.
\end{definition}

\begin{lemma}\label{lem:superquasirandom}
  For every $\gamma>0$ and every integer $\Delta\ge 1$ there is $\alpha>0$
  such that if~$G$ is an $\alpha$-quasirandom graph of density $d>\gamma$
  and order $m$, then~$G$ contains an induced
  $(\gamma,\Delta)$-superquasirandom subgraph  of order at least
  $(1-\gamma)m$ and density $d\pm \gamma$.
\end{lemma}
\begin{proof}
  We can assume that $\gamma<\frac12$. Let $\alpha'$ be given by
  Lemma~\ref{lem:bad} for input parameters $\gamma'=\gamma d^\Delta/200$ and~$\Delta$,
  and set $\alpha=\min\{\alpha',d\gamma/(800\cdot 2^\Delta)\}$.  Now suppose that
  $G$ is an $\alpha$-quasirandom graph of density~$d$ and order $m$. By
  Lemma~\ref{lem:bad}, we have $|\bad_{\gamma',\Delta}(G)|\le
  \gamma'm$.

  We claim that the induced subgraph~$G'$ on the vertex set
  $V'=V\setminus\bad_{\gamma',\Delta}(G)$ satisfies the assertion of the
  lemma. Indeed, $|V'|\ge(1-\gamma')m$ and since~$G$ is
  $\alpha$-quasirandom the density~$d'$ of~$G'$ satisfies
  $d'=(d\binom{|V'|}{2}\pm\alpha n^2)/\binom{|V'|}{2}=d\pm
  4\alpha=d\pm\gamma$. It
  remains to show that $\bad_{\gamma,\Delta}(G')=\emptyset$.
  By the definition of~$G'$, for each $v\in V'$ and $p\le\Delta$
  all but at most $\gamma'\binom{|V|}{p-1}$ sets
  $\{v_1,\ldots,v_{p-1}\}\in\binom{V'}{p-1}$ are such that
  $\{v,v_1,\dots,v_{p-1}\}$ is not $\gamma'$-bad in~$G$. But such sets
  $\{v,v_1,\dots,v_{p-1}\}$ are not $\gamma$-bad in~$G'$ either because
  \begin{equation*}\begin{split}
    |\neighbor_{G'}(v,v_1,\dots,v_{p-1})|&=|\neighbor_G(v,v_1,\ldots,v_{p-1})|\pm |\bad_{\gamma',\Delta}(G)|
    =(1\pm\gamma')d^p m\pm\gamma' m\\
    & =(1\pm\tfrac1{100}\gamma)d^p m
    =(1\pm\tfrac1{100}\gamma)(d'\pm 4\alpha)^p(1\pm\gamma')|V'|\\
    &=\big(1\pm 10(\tfrac1{100}\gamma+ 2^p\cdot 4\alpha\tfrac{1}{d'}+\gamma')\big)(d')^p|V'|
    =(1\pm \gamma)(d')^p|V'|\;,
  \end{split}\end{equation*}
  where we use $2^p\cdot 4\alpha\frac{1}{d'}\le\gamma/100$.
  Hence $\bad_{\gamma,\Delta}(G')=\emptyset$.
\end{proof}

The next easy lemma asserts that very dense graphs are
quasirandom.

\begin{lemma}\label{lem:densequasirandom}
For any $\alpha>0$ there exist
$n_0=n_{L\ref{lem:densequasirandom}}(\alpha)$ such that the
following holds for any $n\ge n_0$. Suppose that $G$ was obtained from the
complete graph~$K_n$ by deleting at most $n$ edges. Then~$G$ is
$\alpha$-quasirandom.
\end{lemma}

\subsection{Homomorphisms}
Let~$H$ and~$G$ be graphs.
A \DEF{homomorphism}~$h$ from~$H$ to~$G$ is an edge-preserving map
from~$V(H)$ to~$V(G)$, i.e., for every $xy\in E(H)$ we have
$h(x)h(y)\in E(G)$. 
By $h\colon H\to G$ or simply $H\to G$ we denote the fact that there
is a homomorphism~$h$ from~$H$ to~$G$.
Moreover, we write \marginpar{$V(h)$}$V(h)=\{h(v)\colon v\in V(H)\} \subset V(G)$
for the image of $h$, and
\marginpar{$E(h)$}$E(h)=\{h(x)h(y)\colon xy\in E(H) \}\subset E(G)$ for the image of
the edges of $H$. 

We say that a map $h$ is a \DEF{partial homomorphism} of $H$ to $G$ if
there exists a set $Y\subset V(H)$ such that $h$ is a homomorphism of $H-Y$
to $G$. The set $Y$ is called vertices \DEF{skipped} by $h$. We define \marginpar{$V(h)$}$V(h)=\{h(v)\colon v\in V(H)-Y\} \subset V(G)$, and \marginpar{$E(h)$}$E(h)$ analogously. We denote the fact that $h$ is a partial homomorphism by \marginpar{$H\rightsquigarrow G$}$h: H\rightsquigarrow G$.

In the language of homomorphisms, a packing of a family
$(H_1,\dots,H_k)$ of graphs into $G$ is a family of injective
homomorphisms $(h_i\colon H_i\to G)_{i\in[k]}$ with mutually
disjoint images of the edge sets.

Let $(h_i)_{i\in[k]}$ be a family of homomorphisms~$h_i\colon H_i\to
G$ with $i\in[k]$ (we assume that the graphs $H_i$ live on different vertex sets). Then the
\emph{union} $\bigcup_{i\in[k]}h_i$ of $(h_i)_{i\in[k]}$ is the map
$h\colon\bigcup_{i\in[k]}V(H_i)\to G$ defined by $h(x)=h_i(x)$ for
all vertices $x\in V(H_i)$ and all $i\in[k]$. 


\subsection{Trees}
The pair $(F,X)$ is a \DEF{rooted forest} if $F$ is a forest and $X\subset
V(F)$ contains exactly one vertex of every tree $C\in\component(F)$
of~$F$, which we call \DEF{root} of~$C$. If~$F$ is a tree with root~$x$ then we
also write $(F,x)$ for $(F,X)=(F,\{x\})$ and say that $(F,x)$ is a \DEF{rooted
tree}.
In a rooted forest $(F,X)$ we can speak of \DEF{children},
\DEF{parents}, \DEF{ancestors}, and \DEF{descendants} of vertices.
For a vertex $y$, we let $F(y)$ be the maximal subtree
of~$F$ with root~$y$.

\subsubsection{Cutting trees}
The central notion of this section is that of a $\rho$-balanced $r$-level
partition defined below. 
\begin{definition}[balanced level partition]
Given a rooted tree $(T,x)$, we say that a partition $\mathcal
P=(L^1,\ldots,L^r)$ of $V(T)$ is a \DEF{$\rho$-balanced $r$-level
partition} if
\begin{enumerate}[label=\abc]\label{def:blp}
  \item\label{def:blp:a}
    $|L^i| = (1\pm\rho/2)\frac{v(T)}{r}$
    for every $i\in [r]$, and
  \item\label{def:blp:b} 
    for each $i\in[r]$, the parent of each non-root vertex in $L^i$ lies in the set $\bigcup_{j\le i}L^j$.
\end{enumerate}
The forest $T[L^i]$ is called \DEF{level} $i$ of
the partition $\mathcal P$. For a vertex~$y$ of $T[L^i]$ or a tree
$C\in\component(T[L^i])$ we say that~$y$ or~$C$ are in level~$i$ of $\cP$.
\end{definition}

The following lemma states that bounded-degree trees have balanced level
partitions with a bounded number of components in each level.

\begin{lemma}\label{lem:cuttree}
  Let $(T,x)$ be a rooted tree with maximum degree at most $\Delta$ and $v(T)\ge
  \frac{4\Delta r}{\rho}$ with $0<\rho < \frac{1}{4r} $ and $r\in\mathbb N$. Then there
  is a $\rho$-balanced $r$-level partition of $(T,x)$ such that every level has
  at most $\frac{8\Delta}{\rho}$ components. 
\end{lemma}
\begin{proof}
  Let $\xi=\rho/(2r)$.
  We first partition~$T$ into a family $\mathcal C=(C_i)_{i\in[\ell]}$ (for some $\ell$) 
	of rooted connected
  components $C_i$ of $T$ so that
  \begin{equation}\label{eq:cuttree:Ci}
    v(C_i)\in\big[\tfrac1\Delta \xi v(T)-1,\xi v(T)\big]
    \text{ for all $i\in[\ell-1]$}
    \qquand
    v(C_\ell)\le \xi v(T)\,.
  \end{equation}
  Clearly, such a partition can be obtained by the following simple algorithm. 
Starting with the root, always proceed downwards in the tree order, at each step choosing the child $y$ maximising $|F(y)|$ until 
$|F(y)|\le \xi v(T)$ is satisfied for the first time. This gives the upper bound 
	in~\eqref{eq:cuttree:Ci}, and since this upper bound was not satisfied when we were looking at the parent of $y$, the lower bound in~\eqref{eq:cuttree:Ci} must also be satisfied.   
	%
	%
	In this  way, we obtain the first component~$C_1=C$, which we cut off from~$T$ and
  then repeat in order to obtain the remaining components.

  We now inductively define the sets $L^1,\dots,L^r$ where each set $L^i$ will
  be the union $L^i=\bigcup_{C\in\cC_i}V(C)$ for a suitable set $\cC_i$ of components.
  Suppose we have already chosen $L^1,\dots,L^{i-1}$ together with 
	$\cC_1,\dots,\cC_{i-1}$. Now choose $\cC_i \subset \cC\setminus \bigcup_{j<i} \cC_j$ 
	satisfying the following two properties:
	\begin{itemize}
	\item
	for every $C\in\cC_i$ and for every $C'\in\cC\setminus \bigcup_{j<i} \cC_j$ 
	that is above $C$ in the tree order, we must have $C'\in \cC_i$,
	\item
	we have $|L_i|=\sum_{C\in\cC_i}|V(C)|=(\frac 1r \pm \xi) v(T) 
	= (1\pm \frac{\rho}{2}) \frac{v(T)}{r}$.	
	\end{itemize}
This choice of $\cC_i$ is clearly possible by the upper bound given 
in~\eqref{eq:cuttree:Ci}.
	Both Conditions~\ref{def:blp:a} and~\ref{def:blp:b}
in Definition~\ref{def:blp} are satisfied by construction and 
	it remains to bound the number $|\cC_i|$ of components in each level $T[L^i]$.
  First observe that due to the assumption $v(T) \ge 4\Delta r/\rho$, we know that 
	\begin{equation}
	\frac{\xi v(T)}{2\Delta} = \frac{\rho v(T)}{4\Delta r} \ge 1\;.
	\label{eq:vTbig}
	\end{equation}
	Therefore we get
  \begin{equation*}
  \begin{split}
    |\cC_i|\le \frac{|L^i|}{\min_{j\in[\ell-1]}|C_j|} + 1
        \le \frac{(\tfrac{1}{r}+\xi)v(T)}{\tfrac{1}{\Delta}\xi v(T)-1}+1
				\leBy{\eqref{eq:vTbig}} \frac{(\tfrac{1}{r}+\xi)v(T)}{\tfrac{1}{2\Delta}\xi v(T)}+1
				= \frac{2\Delta}{\xi r} + 2\Delta +1 \le \frac{8\Delta}{\rho} \;,
  \end{split}
  \end{equation*}
  and hence the partition $V(T)=L^1\dcup\cdots\dcup L^r$ satisfies all requirements of the lemma.
\end{proof}

\subsection{Probabilistic tools} We write $\mathrm{Be}(p)$ for the Bernoulli distribution with success probability $p$, and we write $\mathrm{Bin}(p,n)$ for the binomial distribution with $n$ trials and success probability $p$.

We will use the following two versions of the Chernoff bound~\cite[(2.9) and
(2.12)]{JLR:RandomGraphs}. Let~$X\in\mathrm{Bin}(n,p)$, and  $\mu\ge\Exp [X]$, $\delta\in(0,\frac32)$, $t>0$. We have that
\begin{align}
\label{eq:Ch2}
\Prob\left[X\ge (1+\delta) \cdot \mu \right]&\le 2\exp\left(-\delta^2\mu/3\right) \quad\text{and}\\
\label{eq:Ch1}
\Prob\left[X\ge \mu+t\right]&\le \exp\left(-\frac{2t^2}{n}\right)\;.
\end{align}
Moreover, for every $\delta'>1$ and every $t\in\mathbb{R}$ with $t\ge \delta'\Exp [X]$ 
there exists $\delta''>0$ such that 
\begin{align}
\label{eq:Ch3}
\Prob\left[X\ge t\right]&\le \exp\left(-\delta''t\right)\;.
\end{align}
Obviously, these bounds also hold for random variables which are stochastically dominated by~$X$.

Suppose that $\Omega=\prod_{i=1}^k \Omega_i$ is a product probability space. A measurable function $f:\Omega\rightarrow \mathbb R$ is said to be \DEF{$C$-Lipschitz} if for each $\omega_1\in\Omega_1,\omega_2\in\Omega_2,\ldots,\omega_i,\omega'_i\in\Omega_i,\ldots,\omega_k\in\Omega_k$ we have $$|f(\omega_1,\omega_2,\ldots,\omega_i,\ldots,\omega_k)-f(\omega_1,\omega_2,\ldots,\omega'_i,\ldots,\omega_k)|\le C\;.$$
McDiarmid's Inequality, \cite{McDiarmid:BoundedDiff} states that Lipschitz functions are concentrated around their expectation.
\begin{lemma}[McDiarmid's Inequality]\label{lem:McDiarmid}
Let $f:\Omega\rightarrow \mathbb R$ be a $C$-Lipschitz function defined on a product probability space $\Omega=\prod_{i=1}^k \Omega_i$. Then for each $t>0$ we have
$$\Prob\big[|f-\Exp[f]|>t\big]\le 2\exp\left(-\frac{2t^2}{C^2 k}\right)\;.
$$
\end{lemma}
We shall also need Talagrand's Inequality, in a version as in~\cite[Theorem~2]{McDRee:Talagrand}.\footnote{All the applications of McDiarmid's Inequality below could actually be replaced by Talagrand's Inequality. However the former has  assumptions that are easier to check and a conclusion that is cleaner.} For a function $f:\Omega\rightarrow \RR$ in a probability space $\Omega=\prod_{i=1}^k \Omega_i$, we say that \emph{values $\omega_{i_1}\in\Omega_{i_1},\ldots,\omega_{i_p}\in\Omega_{i_p}$ certify\marginpar{certify} that $f\ge \Lambda$} if for each choice of $(\omega_j\in \Omega_j)_{j\in [k]\setminus\{i_1,\ldots,i_p\}}$ we have that $f(\omega_1,\ldots,\omega_k)\ge \Lambda$.
\begin{lemma}[Talagrand's Inequality]\label{lem:Talagrand}
Let $f:\Omega\rightarrow [0,+\infty)$ be a $C$-Lipschitz function defined on a product probability space $\Omega=\prod_{i=1}^k \Omega_i$. Suppose also that there exists a constant $c>0$ such that if we have $\Lambda>0$ and $\omega_1\in\Omega_1,\ldots,\omega_k\in\Omega_k$ such that $f(\omega_1,\ldots,\omega_k)\ge \Lambda$ then there is a set of at most $c\Lambda$ values that certify $f\ge \Lambda$. Then for each $t>0$ we have
$$\Prob\left[f\ge \Exp[f]+t\right]\le \exp\left(-\frac{t^2}{2cC^2 (\Exp[f]+t)}\right)\;.
$$
\end{lemma}

Next, we introduce Suen's inequality (\cite{Suen:Concentration}, see also \cite[p.~128]{AS08}). Let $\{B_i\subset \Omega\}_{i\in I}$ be a finite collection of events in an arbitrary probability space $\Omega$. A \DEF{superdependency graph} for $\{B_i\}_{i\in I}$ is an arbitrary graph on the vertex set $I$ whose edges satisfy the following. Let $I_1,I_2\subset I$ be two arbitrary disjoint sets with no edge crossing from $I_1$ to $I_2$. Then any Boolean combination of the events $\{B_i\}_{i\in I_1}$ is independent of any Boolean combination of the events $\{B_i\}_{i\in I_2}$. In this setting (and only in this setting) we write \DEF{$i\sim j$} to denote that $ij$ forms an edge. 

Suen's Inequality allows us to approximate $\Prob[\bigwedge \overline{B_i}]$ by $\prod \Prob[\overline{B_i}]$.
\begin{lemma}[Suen's Inequality]\label{lem:Suen}
Using the above notation, and writing $M=\prod \Prob[\overline{B_i}]$, we have
$$\left|\Prob\Big[\bigwedge \overline{B_i}\Big]-M\right|\le M \cdot \left(
\exp\big(\sum_{i\sim j}\nu_{i,j}\big)-1\right) \;,$$
where $$\nu_{i,j}=\frac{\Prob[B_i\wedge B_j]+\Prob[B_i]\Prob[B_j]}{\prod_{\text{$\ell\sim i$ or $\ell\sim j$}}(1-\Prob[B_\ell])}\;.$$
\end{lemma}

\section{Almost packings via the nibble method}\label{sec:ExistsAlmostPacking}

In this section, we prove the almost packing lemma
(Lemma~\ref{lem:ExistsAlmostPacking}). 

\subsection{Outline of the proof of Lemma~\ref{lem:ExistsAlmostPacking}}

Given an $(n,\Delta)$-family of trees we want to find an almost packing
into $K_{(1+\eps)n}$. Our first step is to prepare the trees 
(see Section~\ref{sec:grouping}): We start by grouping all trees but the
exceptional tree~$T_0$ according to their sizes into $c=50/\eps$ many groups
so that trees in each group have almost the same number of vertices. The
reason behind this is that one of our goals is to get good bounds on the
quasirandomness of the host graph after each packing round of the nibble method, and for
obtaining these bounds we need a very fine-grained control over the
sizes of the forests embedded in one round. Since our trees can be very
different in size, however, we group them as described and show that
quasirandomness is maintained for each group individually (hence also in
total).
Unfortunately though, even the difference in tree sizes within one group
(which are at most $n/2c$) is
too big for the precision that we need for our quasirandomness
bounds. We resolve this issue by attaching a small path (of length at most $n/2c=\eps n/100)$
to each tree, we can guarantee that in each group~$i\in[c]$ all
trees~$T_{i,s}$ with $s\in[k_i]$ are actually of exactly the same
size. Observe that in total this adds at most $\eps n^2/50$ edges to our
tree family, hence the resulting family in total still has less edges
than $K_{(1+\eps)n}$.
Next, we use Lemma~\ref{lem:cuttree} to obtain a $\rho$-balanced
$r$-level partition of each tree~$T_{i,s}$ such that each level~$F^j_{i,s}$
with $j\in[r]$ forms a forest with constantly many components and all the
levels are of similar size. The resulting difference in forest sizes within
one group now is sufficiently small for the precision that we need for our
quasirandomness bounds.

Our second step (see Section~\ref{sec:treezero}) is to remove a copy
of~$T_0$ from~$K_{(1+\eps)n}$. The
resulting graph is still $\alpha_1$-quasirandom for arbitrarily
small~$\alpha_1$. Our third step is to almost pack the remaining trees
in~$r$ rounds. In round~$j$ we embed level~$F^j_{i,s}$ of tree~$T_{i,s}$
for all $i\in[c]$ and $s\in[k_i]$. That this is possible is guaranteed by the nibble lemma,
Lemma~\ref{lem:newnibble} (see Section~\ref{sec:nibble}). This lemma states
that in an $\alpha_j$-quasirandom graph~$G_j$ we can find partial
homomorphisms from our levels~$F^j_{i,s}$ to~$G_j$ such that these
homomorphisms produce an almost packing of~$F^j_{i,s}$. At the end of round~$j$ we
remove from~$G_j$ all edges used in images of any~$F^j_{i,s}$.
Lemma~\ref{lem:newnibble} also guarantees that the resulting
graph~$G_{j+1}$ is still quasirandom (albeit with worse parameters), hence
we can continue with the next round.

\subsection{Constants}
\label{sec:constants}

We now start the proof of Lemma~\ref{lem:ExistsAlmostPacking}.
Suppose that $\eps>0$ and $\Delta\in\mathbb N$ are given. 
Set $c=\frac{50}{\epsilon}$ and let\marginpar{$c,r,\beta_r$}
\begin{align}\label{eq:constants:r}
  r=\frac{1000\Delta^2}{\epsilon^{10\Delta}}\quad\mbox{ and }\quad
  \beta_r= \epsilon^2/100.
\end{align}
We recursively define 
\marginpar{$\alpha_j,\beta_j$}
$\alpha_r,\beta_{r-1},\alpha_{r-1},\dots,\beta_1,\alpha_1$ by setting  
\begin{equation}\label{eq:constants:betaalpha}
\alpha_j=\alpha_{\subsc{L\ref{lem:newnibble}}}(\eps,\beta_j,c,\Delta)
\text{\quad and \quad}
\beta_{j-1}=\alpha_j\;,  
\end{equation}
using Lemma~\ref{lem:newnibble} below. Note that we have that
$\alpha_1\le \beta_1=\alpha_2\le \beta_2=\alpha_3 \le \dots = \alpha_r\le\beta_r$.

Finally, let\marginpar{$\rho, n_0$}
\begin{align}\label{eq:constants:rho}
\rho= \min\{\frac{1}{4r},\alpha_1\}\quad\mbox{ and }\quad
n_0= \max\{ \frac{8\Delta r}{ \rho\alpha_1}, n_{\subsc{L\ref{lem:densequasirandom}}}(\alpha_1), n_1, n_2,\ldots,n_r \}\;,
\end{align}
where $n_i=n_{\subsc{L\ref{lem:newnibble}}}(\eps, \beta_i, c, \Delta, \alpha_i, r)$.
\subsection{Preparing the trees} 
\label{sec:grouping}

Now that we have chosen $n_0$ as required by Lemma~\ref{lem:ExistsAlmostPacking}, 
consider an $(n,\Delta)$-tree family $\mathcal T$ 
and let $T_0\in \mathcal T$ be the exceptional tree of order at most $n/2$ (if it exists). 
In the following embedding procedure $T_0$ will be treated separately.

We group the other trees in $\mathcal T$ according to their order. 
For $i\in[c]$ let $T_{i,1},T_{i,2},\ldots,T_{i,k_i}$ be the trees of $\mathcal T$ whose order is in the interval $\left(\frac n2+(i-1)\cdot \frac{\epsilon n}{100},\frac n2+i\cdot \frac{\epsilon n}{100}\right]$.
We append to an arbitrary leaf of each tree $T_{i,s}$ a path with exactly 
$\frac n2+i\cdot \frac{\epsilon n}{100} -v(T_{i,s})$ edges. 
As a result, each modified tree  $T_{i,s}$ has order exactly $\frac n2+i\cdot \frac{\epsilon n}{100}$. 
Since $\mathcal T$ contains at most $2n$ trees, this added at most 
$\frac{\varepsilon n^2}{50}$ edges to the total number of edges in $\mathcal T$ and thus
\begin{equation}
\label{eq:totalEdgesAfterAdding}
\sum_{i\in [c], s\in [k_i]} e(T_{i,s}) \le \binom{n}{2} + \frac{\eps n^2}{50}.
\end{equation}
The order and the maximum degree of the trees are still upper-bounded by $n$ and $\Delta$, respectively. 
For $i\in [c]$ we now let 
\marginpar{$n_i$}
\begin{equation}
\label{eq:def:ni}
n_i=\frac{n}{2r} + i \frac{n}{2cr}=\frac{n}{2r} + i \frac{\eps n}{100r} = \frac{v(T_{i,s})}{r} .
\end{equation}

We slice the trees into~$r$ levels as follows. We pick an
arbitrary root \marginpar{$x_{i,s}$}$x_{i,s}$ for each tree $T_{i,s}$ with $i\in [c],\; s\in [k_i]$. For
all $i\in[c],\, s\in [k_i]$ we apply Lemma~\ref{lem:cuttree} to the rooted tree
$(T_{i,s},x_{i,s})$. Since 
\begin{equation*}
  v(T_{i,s})>\frac{n}2
    \ge\frac{n_0}2
    \geByRef{eq:constants:rho} \frac{4\Delta r}\rho 
		\quad \text{ and } \quad
		\rho \leByRef{eq:constants:rho} \frac{1}{4r} ,
\end{equation*}
we obtain a $\rho$-balanced $r$-level partition
\marginpar{$\mathcal P_{i,s}=(L_{i,s}^1,\ldots,L_{i,s}^r)$}$\mathcal P_{i,s}=(L_{i,s}^1,\ldots,L_{i,s}^r)$ of $(T_{i,s},x_{i,s})$
such that every level of $\mathcal P_{i,s}$ has at most
$8\Delta/\rho$ components. 
Finally, we use these partitions to define rooted forests
$(F_{i,s}^j,X_{i,s}^j)$ with $i\in[c],\, s\in[k_i]$ and $j\in[r]$ as follows.
Let \marginpar{$F_{i,s}^j$}$F_{i,s}^j=T_{i,s}[L_{i,s}^j]$ be the level~$j$ of the partition~$\cP_{i,s}$
and let\marginpar{$n_{i,s}^j$}
\begin{equation}
\label{eq:def:nisj}
n_{i,s}^j = v(F_{i,s}^j) = |L_{i,s}^j| \By{\text{Def }\ref{def:blp}}{=} 
(1\pm \frac{\rho}{2}) \frac{v(T_{i,s})}{r} 
\By{(\ref{eq:def:ni})}{=} (1\pm \frac{\rho}{2}) n_i. 
\end{equation}
Using the fact that $\rho\le 1/(4r)$ by~\eqref{eq:constants:rho}, we obtain that
\begin{equation}\label{eq:proof:nij}
  \frac{n}{4r}\le n_{i,s}^j\le \frac{2n}{r}\,.
\end{equation}

Let the root set~$X_{i,s}^j$\marginpar{$X_{i,s}^j$} be obtained by
considering~$F_{i,s}^j$ as a rooted subforest of the rooted tree
$(T_{i,s},x_{i,s})$, that is, $X_{i,s}^1=\{x_{i,s}\}$, and for $j>1$, $X_{i,s}^j$ is composed of the vertices of every component of~$F_{i,s}^j$ that are the closest to $x_{i,s}$.
Lemma~\ref{lem:cuttree} guarantees that for every $i\in[c],\, s\in[k_i]$ and $j\in[r]$ we have
\begin{equation}\label{eq:NCB}
  |X_{i,s}^j|
  \le \frac{8\Delta}{\rho}
  \leByRef{eq:constants:rho} \alpha_1\frac{n}{r}
  \;.
\end{equation}


\subsection{Embedding \texorpdfstring{$T_0$}{T0}}
\label{sec:treezero}

Our embedding procedure now starts by embedding~$T_0$ arbitrarily into
$K_{(1+\epsilon)n}$. By Lemma~\ref{lem:densequasirandom} the resulting graph
\begin{equation}\label{eq:embedding:G1}
  \text{$G_1=(V,E_1)=K_{(1+\epsilon)n}-T_0$ is $\alpha_1$-quasirandom.}
\end{equation}

\subsection{The nibble lemma}
\label{sec:nibble}

For almost packing the remaining trees we use a nibble method, that
is, we proceed in rounds and embed in each round one level of each tree.
The setting of Lemma~\ref{lem:newnibble}, which captures
one round of the nibble procedure, is as follows.

We have a quasirandom host graph $G=(V,E)$ and a family
$(F_{i,s},X_{i,s})_{s\in [k_i],\,i\in[c]}$ of rooted forests that we want
to pack into $G$, one sub-forest~$F_{i,s}$ for each tree~$T_{i,s}$ to be
packed. In addition, we are given for each $i\in[c], s\in [k_i]$ a
set~$U_{i,s}\subset V$ of \emph{forbidden vertices} for the embedding
of~$F_{i,s}$. The set~$U_{i,s}$ contains vertices of~$G$ that were used for
the embedding of vertices of~$T_{i,s}$ in earlier rounds.

It is the quasirandomness of~$G$ that will enable us to almost pack the
forests $F_{i,s}$. While doing so, however, we need to keep in mind that
there are future embedding rounds to come. Therefore we cannot embed the
forest just somehow, but we have to assert that certain invariants
are maintained. One of these invariants is clearly the
quasirandomness of the part of the host graph that remains after the
embedding (Property~\ref{cond:GStillQuasirandom}). In addition we need to
guarantee that the embedding of the different forests~$F_{i,s}$ is
distributed ``fairly'' over the vertices of~$G$. To this end we require
that the sets~$U_{i,s}$ are well spread over~$G$ and our goal is to
maintain this property for the next embedding round
(Property~\ref{cond:UStillSmallLoad}).  For this we need a concept which
measures whether the sets $U_{i,s}$ are distributed in a sufficiently
random-like manner over the vertex set~$V$.

\begin{definition}[load]\label{def:load}
Consider a graph $G=(V,E)$ with $m=|V|$ and two vertices $v,w\in V$, and 
let $\cW=(W_{s})_{s\in [k]}$ be a collection of subsets of $V$.
\marginpar{$\load(v,w,\cW)$}
\marginpar{$\mu(\cW)$}
\marginpar{$\sigma(\cW)$}
\begin{align*}
\load(v,w,\cW) &= \left| \left\{ s\in [k] \st W_{s} \cap \{v,w\} \neq \emptyset \right\} \right|, \\
\mu(\cW) &= \frac{1}{\binom{m}{2}}\sum_{\{v',w'\}\in \binom{V}{2}} \load(v',w',\cW),\\
\sigma(\cW) &= \sum_{\{v',w'\}\in \binom{V}{2}} \left(\load(v',w',\cW) - \mu(\cW)\right)^2.
\end{align*}
We say that $\cW$ is \DEF{$(\alpha,\ell)$-homogeneous} if $\sigma(\cW)\le \alpha \ell^4$, and for each $s,s'\in[k]$ we have $||W_{s}|-|W_{s'}||\le \alpha \ell$.
\end{definition}

In the proof of Lemma~\ref{lem:newnibble} we will maintain these
invariants by embedding the forests~$F_{i,s}$ randomly, that is, we construct
random partial homomorphisms $h_{i,s}\colon F_{i,s} \rightsquigarrow G$. 
The mappings $h_{i,s}$ do not embed the vertices in $X_{i,s}$, and there will be another family of sets, denoted by $Y_{i,s}$ and called the \DEF{skipped} vertices, that are left unembedded. Thus the 
$h_{i,s}\colon F_{i,s}-(X_{i,s}\cup Y_{i,s}) \to G$ are homomorphisms. 
However, they do not necessarily form a proper packing of 
$F_{i,s}-(X_{i,s}\cup Y_{i,s})_{s\in [k_i],\,i\in[c]}$ into~$G$,
because they may fail to be injective or pairwise edge-disjoint. 
In order to measure this shortcoming, we introduce various types of
\emph{collisions}, which we describe in the following definition.

\begin{definition}[colliding and skipped vertices]\label{def:collisions}
In the setting above, suppose that $h_{i,s}: F_{i,s}-(X_{i,s}\cup Y_{i,s})\rightarrow G$ are homomorphisms.
We say that a vertex $y\in V(F_{i,s})$
is in a \DEF{vertex collision} or that~$y$ is \DEF{colliding}, if there exists a vertex $z\in V(F_{i,s})\setminus
\{y\}$ such that $h_{i,s}(y)=h_{i,s}(z)$. We define\marginpar{$\VC_{i,s}$}
\[\VC_{i,s}=\{y\in V(F_{i,s})\::\: y\text{ is colliding}\}\;.\]
We say that an edge $xy\in E(F_{i,s})$ is \emph{colliding} if there is some
$(i',s')\neq (i,s)$ with $x'y'\in E(F_{i',s'})$ such that $h_{i,s}(x,y)=h_{i's'}(x',y')$. A vertex $y\in V(F_{i,s})$ is in an \DEF{edge collision} if there is $x\in V(F_{i,s})\setminus
\{y\}$ such that $xy$ is colliding.
We define\marginpar{$\EC_{i,s}$}
\[\EC_{i,s}=\{y\in  V(F_{i,s})  \::\: y\text{ is in an edge collision}\}\;.\]
We say a vertex $x\in \bigcup_{i,s}V(F_{i,s})$ is \DEF{faulty} if $x\in \bigcup_{i,s} (\VC_{i,s}\cup \EC_{i,s})$. 

For a vertex~$v\in V$ the \emph{vertices mapped to~$v$ with faulty
neighbours} are \marginpar{$\FN(v)$}
\begin{equation*}
  \FN(v)=
    \bigcup_{i,s}\big\{x\in h^{-1}_{i,s}(v)
    \colon\, \exists xy\in E(F_{i,s})\text{ such that $y$ is faulty}\big\}\;,
\end{equation*}
the \emph{vertices mapped to~$v$ with skipped neighbours} are\marginpar{$ \YN(v)$}
\begin{equation*}
  \YN(v)=
    \bigcup_{i,s}\big\{x\in h^{-1}_{i,s}(v)
    \colon\, \exists xy\in E(F_{i,s})\text{ such that $y\in Y_{i,s}$}\big\}\;,
\end{equation*}
and the \emph{vertices mapped to~$v$ with root neighbours}  are\marginpar{$\XN(v)$}
\begin{equation*}
  \XN(v)=
    \bigcup_{i,s}\big\{x\in h^{-1}_{i,s}(v)
    \colon\, \exists xy\in E(F_{i,s})\text{ such that $y\in X_{i,s}$}\big\}\;,
\end{equation*}
\end{definition}

Lemma~\ref{lem:newnibble} now asserts that we only have a small number of these collisions. 
As we will show after stating the lemma, this implies that we get an almost embedding.

\begin{lemma}[Nibble Lemma]\label{lem:newnibble}
For every $\epsilon,\beta>0$, and $c, \Delta\in\mathbb N$,  
there exists $0<\alpha\le \beta$ so that for every integer $r$ there exists $n_0$ such that for every $n\ge n_0$ the following is true.

We assume that we are given a family of rooted forests
$\cF=(F_{i,s},X_{i,s})_{i\in [c],s\in [k_i]}$ with
$n/2\le\sum_{i=1}^{c}k_i\le 2n$, $|X_{i,s}| \le \alpha\frac{n}{r}$,
$n_{i,s}=v(F_{i,s}) = (1\pm\alpha)n_i$\marginpar{$n_{i,s}$}, where, as before, $n_i=\frac{n}{2r} + i
\frac{n}{2cr}$.
Moreover, we assume that $G=(V,E)$ is an $\alpha$-quasirandom graph with 
$m=|V|=(1+\epsilon)n$
and  density $d>\epsilon$. For each $i\in [c]$, let
$\cU_i=(U_{i,s})_{s\in [k_i]}$ be an $(\alpha,n)$-homogeneous  family with $|U_{i,s}|<n$ for all $s\in[k_i]$. For all $i\in [c],\, s\in [k_i]$ set $V_{i,s}=V\setminus U_{i,s}$.

Then there are sets $Y_{i,s}\subseteq V(F_{i,s})$ and homomorphisms 
$h_{i,s}: F_{i,s} - (X_{i,s} \cup Y_{i,s}) \to G[V_{i,s}]$ for all $i\in [c]$ and  $s\in [k_i]$, with the following properties. For each $i\in [c]$, each $s\in [k_i]$, and each $v\in V(G)$ we have 
\begin{enumerate}[label={\rm (C\arabic{*})}]
\item\label{cond:NumberSkippedVertices}
 $|Y_{i,s}| \le \beta n/r$, 
\item\label{cond:VertexCollisions}
 $|\VC_{i,s}|\le 
{20n}/({\epsilon r^2d^\Delta})$,
\item\label{cond:EdgeCollisions} 
$|\EC_{i,s}|\le 300\Delta n/(\epsilon^2r^2d^{\Delta})$,
\item\label{cond:NCN} 
$|\FN(v)|\le {10^4\Delta^3n}/({\epsilon^3r^2d^{2\Delta}})$,
\item\label{cond:NumberSkippedNeighbours}
$|\YN(v)|\le 
\beta n/r$, 
\item \label{cond:rootneighbours}
$|\XN(v)|\le \beta n/r$,
\item\label{cond:GStillQuasirandom}
the graph $\tilde G= (V,E\setminus \bigcup_{i,s} E(h_{i,s}))$ is $\beta$-quasirandom, and
\item\label{cond:UStillSmallLoad}
for each $i\in [c]$, the family $\tilde\cU_i=(\tilde U_{i,s})_{s\in [k_i]}$ with $\tilde U_{i,s}=U_{i,s}\cup V(h_{i,s})$ is $(\beta,n)$-homogeneous.
\end{enumerate}
%
\end{lemma}

\subsection{Applying the Nibble Lemma to obtain an almost-packing}
Let us first recall what we have achieved so far. In Section~\ref{sec:grouping} we obtained a 
family $\mathcal{F}^j= (F_{i,s}^j,X_{i,s}^j)$ of rooted forests for $j\in [r]$. We can assume that $\sum_i k_i\ge n/2$ (as otherwise, we might add dummy trees to be embedded).
In Section~\ref{sec:treezero} we embedded the tree $T_0$, deleted its edges and ended up 
with an $\alpha_1$-quasirandom graph $G_1=(V,E_1)$. 

Now we set $\mathcal{U}_i^1=(U_{i,s}^1)_{s\in [k_i]}$ where $U_{i,s}^1=\emptyset$ 
for all $i\in [c]$ and $s\in [k_i]$. 
We perform $r$ embedding rounds. For $j=1,\dots,r$, we do the following in round $j$.
We apply Lemma~\ref{lem:newnibble} with parameters 
$\epsilon$, $\beta_j$, $c$, $\Delta$, obtaining $\alpha_j$ and $n_0$.
We then feed to Lemma~\ref{lem:newnibble}
\begin{enumerate}[label={\rm (P\arabic{*})$_j$}, ref={\rm (P\arabic{*})}]
\item the family $\cF^j=(F^j_{i,s},X^j_{i,s})_{i\in [c],s\in [k_i]}$ of rooted forests,
\item\label{itm:induction:G} an $\alpha_j$-quasirandom graph~$G_j=(V,E_j)$ with $|V|=m=(1+\epsilon)n$ and $d_j\binom{m}{2}=|E_j|\ge \frac34\eps n^2$, which implies $d_j\ge\eps$,
\item\label{itm:induction:U} and for each $i\in [c]$ an $(\alpha_j,n)$-homogeneous family $\cU^j_i=(U^j_{i,s})_{s\in [k_i]}$.
\end{enumerate}
Let us now check that the conditions required by Lemma~\ref{lem:newnibble} are met. 
By~\eqref{eq:NCB} we have $|X^j_{i,s}|\le \alpha_1\frac{n}{r}\le \alpha_j\frac{n}{r}$, 
by~\eqref{eq:def:nisj} and the definition of $\rho$
we have $v(F_{i,s}^j) = (1\pm\alpha_j)n_i$. Hence the conditions of Lemma~\ref{lem:newnibble} are satisfied. So we obtain sets $Y^j_{i,s}\subset V(F^j_{i,s})$ and homomorphisms $h^j_{i,s}: F^j_{i,s} - (X^j_{i,s} \cup Y^j_{i,s}) \to G[V^j_{i,s}]$, where $V^j_{i,s}=V\setminus U^j_{i,s}$, with vertex collisions $\VC^j_{i,s}$, edge collisions $\EC^j_{i,s}$, 
faulty neighbours $\FN^j(v)$, skipped neighbours $\YN^j(v)$, and root neighbours $\XN^j(v)$ for every $v\in V$, such that \ref{cond:NumberSkippedVertices}--\ref{cond:UStillSmallLoad} are satisfied. 

We will next argue that we can apply Lemma~\ref{lem:newnibble} again in the next round. For this purpose 
let $G_{j+1}=(V,E_{j+1})=(V,E_j\setminus \bigcup_{i,s}E(h^{j}_{i,s}) )$. 
Since $\beta_j=\alpha_{j+1}$ by~\eqref{eq:constants:betaalpha}, Conclusion~\ref{cond:GStillQuasirandom} implies that $G_{j+1}$ is $\alpha_{j+1}$-quasirandom. 
Moreover, to check the density requirement in~\ref{itm:induction:G}$_{j+1}$, 
\begin{equation*}\begin{split}
  |E_{j+1}| &\ge e(G_1)-\!\!\! \sum_{\substack{i\in[c], s\in[k_i]\\ j\in[r]}}\!\!\! e(F^j_{i,s})
  \ge e(K_{(1+\epsilon)n})-e(T_0)-\!\!\!\sum_{i\in[c],s\in[k_i]}\!\!\!e(T_{i,s}) \\
  &\geByRef{eq:totalEdgesAfterAdding} \binom{(1+\epsilon)n}{2}-(n-1)-\frac{\eps n^2}{50}-\binom{n}{2} 
  \geByRef{eq:constants:rho} \frac34 \epsilon n^2
  \,.
\end{split}\end{equation*}
Let $\cU^{j+1}_i=\tilde\cU^j_i$. By~\ref{cond:UStillSmallLoad} the family $\cU^{j+1}_i$ is $(\beta_j=\alpha_{j+1},n)$-homogeneous.
We conclude that conditions~\ref{itm:induction:G}$_{j+1}$ and~\ref{itm:induction:U}$_{j+1}$ are again satisfied and hence we can apply Lemma~\ref{lem:newnibble} in the next round.

After finishing all~$r$ embedding rounds we define the set $R_{i,s}$ that contains all roots, skipped vertices and vertices in vertex or edge collisions in the tree~$T_{i,s}$,
\begin{equation*}
  R_{i,s}= \bigcup_{j\in[r]}\big(X^j_{i,s} \cup Y^j_{i,s} \cup\VC^j_{i,s}\cup\EC^j_{i,s}\big)
  \,.
\end{equation*}
Let $\tilde h^j_{i,s}$ be the restriction of $h^j_{i,s}$ to $V(F^j_{i,s})\setminus R_{i,s}$ and $\tilde h_{i,s}=\bigcup_{j\in[r]}\tilde h^j_{i,s}$. We will show that $\{\tilde h_{i,s},R_{i,s}\}_{i\in[c],s\in[k_i]}$ is an $(\eps n)$-almost packing of $\cT$ into $K_{(1+\epsilon)n}-T_0$, which will finish the proof of Lemma~\ref{lem:ExistsAlmostPacking}.
  
Indeed, by the definition of the sets $V^j_{i,s}$, the vertex-images of two homomorphisms~$h^j_{i,s}$ and $h^{j'}_{i,s}$ are disjoint, unless $j=j'$. In other words, vertices of different rounds cannot collide. Moreover, by the definition of 
$G_j$, 
the edges of $K_{(1+\epsilon)n}$ used for the embedding in some round do not get used again in a later round. Hence edges of different rounds can also not collide. Since~$h^j_{i,s}$ is a homomorphism from $F^j_{i,s}-(X^j_{i,s}\cup Y^j_{i,s})$ to $G[V^j_{i,s}]$, the set $\VC^j_{i,s}\cup\EC^j_{i,s}$ contains all vertices in vertex and edge collisions of $F^j_{i,s}$, and $X^j_{i,s}$ contains all roots of trees in $F^j_{i,s}$, we conclude that $\{\tilde h_{i,s}\}_{i\in[c],s\in[k_i]}$ is a packing of the family $\{T_{i,s}- R_{i,s}\}_{i\in[c],s\in[k_i]}$ into $K_{(1+\epsilon)n} -T_0$. 

Hence it remains to check conditions~\ref{def:almostpack:R} and~\ref{def:almostpack:RN} of Definition~\ref{def:almostpack}. 
For condition~\ref{def:almostpack:R}, observe that by \eqref{eq:NCB}, 
\ref{cond:NumberSkippedVertices}, \ref{cond:VertexCollisions} and \ref{cond:EdgeCollisions} of Lemma~\ref{lem:newnibble} we have
\begin{equation*}\begin{split}
  |R_{i,s}|&=\sum_{j\in[r]} \big(
    |X^j_{i,s}|+|Y^j_{i,s}|+|\VC^j_{i,s}|+|\EC^j_{i,s}|
  \big) \\
  &\le r \cdot \Big(
    \alpha_1\frac{n}{r} + \beta_r\frac{n}{r}
    + 20\frac{n}{\epsilon r^2 d_r^\Delta} 
    + \frac{300\Delta n}{\epsilon^3r^2d_r^{2\Delta}}
  \Big)\\
  &
  \le \Big(2\beta_r+\frac{320\Delta}{\eps^3rd_r^{2\Delta}}\Big)n
  \le\epsilon n \,,
\end{split}\end{equation*}
where we use $d_r\ge\eps$, and~\eqref{eq:constants:r}. 
For condition~\ref{def:almostpack:RN}, 
let $v\in V(K_{(1+\eps)n})$ be fixed and define
\begin{equation*}
    \RN(v) =
    \bigcup_{i,s}\big\{y\in h^{-1}_{i,s}(v)\colon\, 
    \exists xy\in E(T_{i,s})\text{ such that $x\in R_{i,s}$}\big\}\;.
\end{equation*}
We need to show that $|\RN(v)|\le\eps n$.
The definition of $R_{i,s}$ implies that 
$\RN(v)=\bigcup_j(\FN^j(v)\cup\YN^j(v)\cup\XN^j(v)\big)$
and thus we infer from
\ref{cond:NCN}, \ref{cond:NumberSkippedNeighbours}, \ref{cond:rootneighbours}
of Lemma~\ref{lem:newnibble} that
\begin{equation*}
  |\RN(v)|\le 
  \Big(\frac{10\Delta^3}{\epsilon^3r^2d_r^{2\Delta}}
    + \frac{\beta_r}{r}
    + \frac{\beta_r}{r}
  \Big)rn
  \le
  \Big(\frac{10\Delta^3\eps^{10\Delta}}{\epsilon^3\cdot 1000\Delta\cdot \epsilon^{2\Delta}}
    + \frac{2\eps^2}{100}
  \Big)n
  \le\eps n\,,
\end{equation*}
where again we use $d_r\ge\eps$, and~\eqref{eq:constants:r}.

\section{Proof of the Correction Lemma}\label{sec:almostpack}

In this section, we give a proof of Lemma~\ref{lem:almostpack}.
We consider the graph $K_m$ as a subgraph of $K_{(1+\epsilon) m}$, and set
$W=V(K_{(1+\epsilon) m})\setminus V(K_m)$. We are given trees
$T_1,\dots,T_k$ together with an $(\frac{\epsilon^2 m}{64\Delta})$-almost packing
$(h_i\colon T_i-R_i\rightarrow K_m)_{i\in[k]}$ of these trees into~$K_m$.

In each tree~$T_i$ we choose a root in $V(T_i)\setminus R_i$ and a 
breadth-first search ordering of the vertices of~$T_i$ starting at this root. We
enumerate the vertices $R_i=\{x_{i,1},\dots,x_{i,\ell_i}\}$ according to
this ordering.
Our approach now is to proceed tree by tree, starting with $T_1$, and to
embed the vertices of $R_i$ one by one into $W$ (in this order), so that we obtain a packing
of all trees into $K_{(1+\eps)m}$ in the end. More precisely, for
$i=1,\dots,k$ and $t=1,\dots,\ell_i$ we map the vertex $x_{i,t}$ to a
vertex~$\tilde h_i(x_{i,t})\in W$ using a greedy algorithm, where $\tilde
h_i(x_{i,t})$ must avoid certain forbidden sets, which we now define.

Firstly, $x_{i,t}$ should not be embedded on vertices in~$W$
which are already images of other vertices of~$T_i$, that is, vertices in
\begin{equation*}
  X_{i,t}=\bigcup_{s<t} \big\{\tilde h_i(x_{i,s})\big\} \,.
\end{equation*}
This will guarantee that $\tilde h_i$ is injective.
Secondly, $x_{i,s}$ should not be embedded on a vertex in~$W$ whose edges
to $h_i$-images of $T_i$-neighbours of $x_{i,t}$ have been used already by
a tree $T_j$ with $j<i$. These forbidden vertex sets are captured below by the sets $Y_{i,t}$ (for $T_i$-neighbours of $x_{i,t}$ that are not in $R_i$) and $U_{i,t}$ (for $T_i$-neighbours of $x_{i,t}$ that are in $R_i$). Let
$A_{i,t}= \neighbor_{T_i}(x_{i,t}) \cap \big( V(T_i)\setminus R_i\big)$
be the neighbours of $x_{i,t}$ that have already been embedded by $h_i$ and set
\begin{equation*}
  Y_{i,t}=\{w\in W\colon \exists j<i,y \in A_{i,t}\colon h_i(y)=v \text{
    and } vw\in E(h_j)\cup E(\tilde h_j)\} \,.
\end{equation*}

Thirdly, we do not want
to embed $x_{i,t}$ to vertices contained in ``dangerously'' many used
edges, that is, vertices in the following set~$Z_i$. Let $E_{i,t}$ be the
set of edges in~$\binom{W}{2}$ that have already been used, that is
$E_{i,t}=\bigcup_{j< i} E(\tilde h_j)\cup E\big((\tilde h_i)\upharpoonright_{\{x_{i,1},\dots,x_{i,t-1}\}}\big)$
and set
\begin{equation*}
  Z_{i,t}=\{w\in W\colon w \text{ is contained in at least $\eps m/2$
    edges of } E_{i,t}\} \,.
\end{equation*}
Embedding~$x_{i,t}$ outside~$Z_{i,t}$ will guarantee that the embedding
process can be continued for the $R_i$-neighbours of~$x_{i,t}$.

Finally, let~$x$ be the parent of~$x_{i,t}$
in~$T_i$. If $x\in R_i$ then we have $x=x_{i,s}$ for some $s<t$. We let
\begin{equation*}
  U_{i,t}= \big\{ w\in W\colon \{\tilde h_i(x_{i,s}),w\} \in E_{i,t} \big\}
  = \big\{ w\in W\colon \{\tilde h_i(x_{i,s}),w\} \in E_{i,s} \big\}
  \,,
\end{equation*}
that is, the set of vertices in~$W$ whose edge to the image of $x_{i,s}$
has been used already. The equality holds because after $x_{i,s}$ and
before $x_{i,t}$ we only embed vertices~$x_{i,s'}$ of~$T_i$ and guarantee
that $\tilde h_i(x_{i,s'})\neq \tilde h_i(x_{i,s})$.
If $x\not\in R_i$ we let $U_{i,t}=\emptyset$.

Having defined these forbidden sets we now map $x_{i,t}$ to an arbitrary vertex
\begin{equation*}
  \tilde h_i(x_{i,t}) \in W \setminus \big( X_{i,t}\cup Y_{i,t}\cup Z_{i,t}
  \cup U_{i,t}\big)\,.
\end{equation*}
We claim that this set is not empty. Indeed, we have
$|X_{i,t}|\le|R_i|\le\eps^2 m/(64\Delta^2)$. In addition, in the definition of
$Y_{i,t}$ there are at most~$\Delta$ choices for~$y$ and hence for~$v$. For
a fixed~$v$, Definition~\ref{def:almostpack}\ref{def:almostpack:RN} states
that at most $\eps^2 m/(64\Delta^2)$ vertices~$z$ have been mapped by
$\bigcup_{j\le i} h_j$ to~$v$. Each of these vertices~$z\in V(T_j)$ has at
most~$\Delta$ neighbours mapped by $\tilde h_j$ to some $w\in W$.  Hence
$|Y_{i,t}|\le\Delta\cdot\Delta\cdot\eps^2 m/(64\Delta^2)$.
To get a bound on $|Z_{i,t}|$ we observe that
\begin{equation*}
  |E_{i,t}|\le\sum_{j\le i} e\big(T_j[R_j]\big)
  \le \sum_{j\le i} |R_j|
  \le k\frac{\eps^2 m}{64\Delta^2}\,.
\end{equation*}
Hence, since $k\le 2m$ we obtain
\begin{equation*}
  |Z_{i,t}|\le \frac{2|E_{i,t}|}{\eps m/2}
  \le \frac{4k\eps^2 m}{64\Delta^2 \eps m}
  \le\frac{\eps m}{8}\,.
\end{equation*}
Moreover, $|U_{i,t}|\le\eps m/2$ because $\tilde h_i(x_{i,s})\not\in
Z_{i,s}$. We conclude that
\begin{equation*}
  \big|W \setminus \big( X_{i,t}\cup Y_{i,t}\cup Z_{i,t}
  \cup U_{i,t}\big)\big|
  \ge \eps m-\frac{\eps^2 m}{64\Delta^2}-\Delta^2\frac{\eps^2 m}{64\Delta^2}
  -\frac{\eps m}{8} - \frac{\eps m}{2} 
  >0 \,.
\end{equation*}

It remains to check that, at the end of this procedure, the mappings
$(h_i\cup\tilde h_i)_{i\in[k]}$ form a packing of~$\mathcal{T}$
into~$K_{(1+\eps)m}$. Firstly, each $h_i\cup\tilde h_i$ is injective,
because $h_i$ is injective, $\tilde h_i$ is injective by the definition of~$X_{i,t}$, and $V(h_i)\cap V(\tilde h_i)=\emptyset$. Secondly,
$h_i\cup\tilde h_i$ is edge-preserving because we embed into a complete
graph. Thirdly, we have $E(h_i\cup\tilde h_i)\cap E(h_j\cup\tilde
h_j)=\emptyset$ for each $i>j$. Indeed, $E(h_i)$ and $E(h_j)$ are disjoint
by assumption. $E(h_i)$ and $E(\tilde h_j)$ (and similarly $E(\tilde h_i)$
and $E(h_j)$) are disjoint by the definition of~$Y_{i,t}$. Finally, 
$E(\tilde h_i)$ and $E(\tilde h_j)$ are disjoint by the definition of~$U_{i,t}$.

\section{Limping homomorphisms on quasirandom graphs}\label{sec:limpingWalk}

Let $F$ be a forest with maximum degree $\Delta$ and a given bipartition into \marginpar{primary~v.}\emph{primary vertices} and \marginpar{secondary v.}\emph{secondary vertices}. Let $G=(V,E)$ be an $(\alpha,\Delta)$-superquasirandom graph of density $d$. We now define a \marginpar{limping homomorphism }\emph{limping homomorphism $h$ from~$F$ to~$G$}. This is a random partial homomorphism from $F$ to $G$ whose distribution is described by the following two-step procedure.
\begin{enumerate}[label=\arabic{*}., leftmargin=*]
\item For each primary vertex $x\in V(F)$ we choose uniformly at random
  (u.a.r.)  a vertex $h(x)\in V$.
\item For each secondary vertex $y\in V(F)$ we choose u.a.r.~a real number
  $\tau(y)\in [0,1)$.  To pick $h(y)$, consider the set
  $\{u_1,\dots,u_p\}= h(\neighbor_{F}(y))$.\footnote{Note that $p$ can be strictly smaller than $\deg_F(y)$; this happens when $h$ is not injective on $\neighbor_{F}(y)$.}
	\begin{enumerate}[label=\abc]
	\item If $\{u_1,\dots,u_p\}$ is $\alpha$-bad then $h$ does not map~$y$ anywhere.
      We say that~$h$ \DEF{skips} $y$.
    \item If $y$ is not skipped, let $i=\lfloor\tau(y)\cdot
      \codeg(u_1,\dots,u_p)\rfloor+1$ and define $h(y)$ to be the $i$-th
      vertex in $\neighbor(u_1,\dots,u_p)$ (for which an order was fixed
      prior to the experiment). In other words, we choose $h(y)$ u.a.r. in
      $\neighbor(u_1,\dots,u_p)$. 

      Modelling this uniform random choice by
      $\tau(y)$ will help in the analysis.
	\end{enumerate}
\end{enumerate}
Hence, if we denote the set of primary vertices by~$P$ and the set of
secondary vertices by~$S$, the limping homomorphism is determined by an
element of the probability space 
\begin{equation}\label{eq:productspace}
  \Omega_F=V^P\times[0,1]^S\;.
\end{equation}
This is the product space that we shall use in applications of McDiarmid's Inequality later.

Observe that a limping homomorphism implicitly depends on the parameter
$\alpha$. This parameter will always be clear from the context.

The next three lemmas establish some fundamental properties of limping homomorphisms.

\begin{lemma}\label{lem:probh(x)=v} Suppose that we are given
  $\alpha\in(0,\frac14)$, a tree $F$ of maximum degree at most $\Delta$
  with a bipartition into primary and secondary vertices, and an
  $(\alpha,\Delta)$-superquasirandom graph $G=(V,E)$ of density $d$ and
  with $|V|\ge 4\Delta/d$.

  Let $h$ be the limping homomorphism from $F$ to $G$.  Let $uv\in E$ be an
  arbitrary edge of~$G$, let $x\in V(F)$ be an arbitrary primary vertex,
  let $y\in V(F)$ be an arbitrary secondary vertex and let $\mathcal H$ be
  an arbitrary event describing the placements of all vertices except
  $y$. Then the following statements hold.
  \begin{enumerate}[label=\abc]
  \item \label{it:prob-prim}$\Prob[h(x)=v]=\frac1{|V|}$.
  \item {\label{it:prob-skippedConditioned}}$\Prob[y\text{ is skipped}\:|\: h(x)=v]\le \alpha$.
  \item \label{it:prob-skipped}$\Prob[\text{$y$ is skipped}\,]\le \alpha$.
  \item \label{it:probedge}Suppose that $xy\in E(F)$. Then $\Prob[h(x)=u
    \text{ and } h(y)=v]
    =\frac{\big(1\pm\alpha\left(\frac{2}{d}\right)^\Delta\big)^{\Delta+2}}{d|V|^2}$.
  \item \label{it:prob-sec}$\Prob[h(y)=v]
    =\frac{\big(1\pm\alpha\left(\frac{2}{d}\right)^\Delta\big)^{\Delta+3}}{|V|}$.
  \item \label{it:prob-secnnonskipped}$\Prob[h(y)=v\:|\:\text{$y$ not
      skipped}\,]=\frac{\big(1\pm\alpha\left(\frac{2}{d}\right)^\Delta\big)^{\Delta+5}}{|V|}$.
  \item \label{it:prob-history}$\Prob[h(y)=v\:|\:\mathcal H]\le
    \frac{2}{d^\Delta |V|}$.
  \end{enumerate}
\end{lemma}
\begin{proof}
\ref{it:prob-prim} This follows immediately from the definition of limping homomorphisms.

\smallskip

\noindent
\ref{it:prob-skippedConditioned} The statement is trivially true when $\neighbor_{F}(y)=\{x\}$. Indeed, then $y$ is never skipped. So, let us assume that $|\neighbor_{F}(y)\setminus\{x\}|\ge 1$.

Let us expose the placement of all the
primary vertices of $F$. Let $\{u_1, \ldots, u_p\}=
h(\neighbor_{F}(y))\setminus \{v\}$. Note that $p\ge 1$ almost surely. As~$G$ is
$(\alpha,\Delta)$-superquasirandom $\badness_{\alpha,p}(v) \le \alpha
\tbinom{|V|}{p-1}$ and so we have
%
%
%
%
\[\Prob\left[y  \text{ is skipped}\:|\: h(x)=v\right]=\Prob\left[\text
{$\{u_1, \ldots,u_p,v\}$  is  $\alpha$-bad} \right]\le \alpha\,.\]

\smallskip

\noindent
\ref{it:prob-skipped} We have
$
  \Prob[y \text{ is skipped}]
  =\sum_{w\in V} \Prob[y \text{ is skipped}\:|\: h(x)=w] \cdot \Prob[h(x)=w]
  \le\alpha
$,
by~\ref{it:prob-prim} and~\ref{it:prob-skippedConditioned}.

\smallskip

\noindent
\ref{it:probedge} 
Let~$A$ be the event that~$x$ gets mapped to~$u$, let~$B$ be the event that~$y$ gets mapped to~$v$, let~$C$ be the event that~$y$ is not skipped, and let~$D$ be
the event that~$v$ is in the common neighbourhood of $h(\neighbor_{F}(y)\setminus
\{x\})$. Note that $B\subset C\cap D$. Indeed, the fact that $B\subseteq C$ is clear. If~$y$ is not skipped, it is mapped to the common neighbourhood of $h(\neighbor_{F}(y))$. So, for $B$ to occur, we need~$v$ to be in this common neighbourhood. But then~$v$ is in the common neighbourhood of $h(\neighbor_{F}(y)\setminus \{x\})$ as well. Hence $B\subseteq D$.

Let $\mathcal{E}_q$ be the event that $|h(\neighbor_{F}(y))|=q+1$. As~$D$ and~$A$ are 
independent even if we condition on~$\mathcal E_q$, we have
\begin{align}
\nonumber
\Prob\left[A\cap B|\mathcal E_q\right]&=\Prob\left[A\cap B\cap C\cap D|\mathcal E_q\right]\\
\nonumber
&=\Prob\left[A|\mathcal E_q\right]\cdot \Prob\left[D|A\cap \mathcal E_q\right]\cdot
\Prob\left[C|\mathcal E_q\cap D\cap A\right]\cdot \Prob\left[B|\mathcal E_q\cap C\cap D\cap A\right]\\
\label{eq:putme}
&=\Prob\left[A|\mathcal E_q\right]\cdot \Prob\left[D|\mathcal E_q\right]\cdot
\Prob\left[C|\mathcal E_q\cap D\cap A\right]\cdot \Prob\left[B|\mathcal E_q\cap C\cap D\cap A\right]\;.
\end{align}


We have 
$\Prob\left[A|\mathcal E_q\right]=\Prob\left[A\right]=\frac{1}{|V|}$. 
As $\badness_{\alpha,1}(v)=0$, we get that $\deg(v)=(1\pm \alpha)d|V|$. Consequently, 
$\Prob\left[D|\mathcal E_q\right]=\big((1\pm \alpha)d\big)^q$.
The number of $\alpha$-bad $(q+1)$-sets that contain~$u$ and have the remaining
vertices inside $\neighbor(v)$ is at most $\alpha\binom{|V|}{q}$. As
$|\neighbor(v)|\ge(1-\alpha)d|V|$, the total number of $(q+1)$-sets that
contain~$u$ and have the remaining vertices inside $\neighbor(v)$ is at
least $\binom{(1-\alpha)d|V|}{q}$. We thus get
\[1\ge \Prob\left[C|\mathcal E_q\cap D\cap A\right]\ge
1-\tfrac{\alpha\binom{|V|}{q}}{\binom{(1-\alpha)d|V|}{q}}\ge
1-\alpha\left(\tfrac{2}{d}\right)^\Delta\;,\] where we use
$(1-\alpha)d|V|-q\ge\frac12d|V|$, which follows from $|V|\ge4\Delta/d$.
Finally, if~$y$ is not skipped, then the set $h(\neighbor_{F}(y))$ is not
$\alpha$-bad, implying that \[\Prob\left[B|\mathcal E_q\cap C\cap D\cap
  A\right]=((1\pm\alpha)d^{q+1}|V|)^{-1}\,.\]  Substituting the above estimates into~\eqref{eq:putme}, we get
\begin{align*}
\Prob\left[A\cap B|\mathcal E_q\right]&
=\Prob\left[A|\mathcal E_q\right]\cdot \Prob\left[D|\mathcal E_q\right]\cdot
\Prob\left[C|\mathcal E_q\cap D\cap A\right]\cdot \Prob\left[B|\mathcal E_q\cap C\cap D\cap A\right]\\
&=\frac{(1\pm \alpha)^qd^q\cdot (1\pm\alpha\left(\tfrac{2}{d}\right)^\Delta)}{(1\pm\alpha)d^{q+1}|V|^2}=
\frac{(1\pm\alpha)^\Delta(1\pm 2\alpha)(1\pm\alpha\left(\tfrac{2}{d}\right)^\Delta)}{d|V|^2}\\
&=\frac{\big(1\pm\alpha\left(\frac{2}{d}\right)^\Delta\big)^{\Delta+2}}{d|V|^2}\;.
\end{align*}
As this quantity does not depend on the choice of~$q$, we get the same answer if we condition on the event $\mathcal E_{q'}$, for any $q'\in [\Delta]$. This gives~\ref{it:probedge}.
\smallskip

\noindent
\ref{it:prob-sec} Fix an arbitrary neighbour $z$ of $y$. Since~$z$ is
primary, we have 
$$\Prob[h(y)=v]=\sum_{w\in V\,:\, vw\in E} \Prob[h(y)=v \text{ and } h(z)=w]\;.$$
The above sum has $(1\pm \alpha)d|V|$ summands. The statement then follows from~\ref{it:probedge}.

\smallskip

\noindent
\ref{it:prob-secnnonskipped} We have
\begin{align*}
\Prob[h(y)=v\;|\;y\text{ not skipped}]&=\frac{\Prob[h(y)=v \text{ and } y\text{ not skipped}]}{\Prob[y\text{ not skipped}]}=\frac{\Prob[h(y)=v]}{\Prob[y\text{ not skipped}]}\;.
\end{align*}
Hence we get the claimed bound from~\ref{it:prob-skipped} and~\ref{it:prob-sec}.

\smallskip

\noindent
\ref{it:prob-history} We can expose the entire embedding of $F-y$, and
condition on the event~$\mathcal H$. Now, either the image of the
neighbours of $y$ form an $\alpha$-bad tuple, or they do not. In the former
case, $y$ is skipped, and the event $h(y)=v$ does not occur. In the latter
case, $y$ is chosen uniformly at random inside a set of size at least
$d^\Delta |V|/2$.
\end{proof}

\begin{lemma}\label{lem:AxAy}
  Suppose that we are given $\alpha\in(0,\frac14)$, a forest $F$ of maximum
  degree at most $\Delta$ with a bipartition into primary and secondary
  vertices, and an $(\alpha,\Delta)$-superquasirandom graph $G=(V,E)$ of
  density $d$ and with $|V|\ge 4\Delta/d$.

  Let $h$ be the limping homomorphism from $F$ to $G$.  Let $x,y\in V(F)$
  be two distinct vertices, and $u,v\in V$ be not necessarily
  distinct. Then we have 
  \[
  \Prob[h(x)=u \text{ and }h(y)=v]<\Big(\frac{2}{d}\Big)^{4\Delta^2}\frac{1}{|V|^2}
  \,.
  \]
\end{lemma}
\begin{proof}
  If $x$ and $y$ form an edge, then this follows from
  Lemma~\ref{lem:probh(x)=v}\ref{it:probedge} because
  \begin{equation*}
    \frac{\big(1+\alpha\left(\frac{2}{d}\right)^\Delta\big)^{\Delta+2}}{d}
    =\frac{(d^\Delta+\alpha 2^\Delta)^{\Delta+2}}{d^{\Delta(\Delta+2)+1}}
    \le\frac{2^{\Delta(\Delta+2)}}{d^{4\Delta^2}}
    \le\Big(\frac{2}{d}\Big)^{4\Delta^2} \,.
  \end{equation*}
  If $x$ and $y$ are in
  different components, or the path from $x$ to $y$ contains at least two
  primary vertices, then~$h(x)$ and~$h(y)$ are independent, and thus the
  claim follows from Lemma~\ref{lem:probh(x)=v}\ref{it:prob-prim}
  and~\ref{it:prob-sec} and a similar calculation as in the previous case. 

  Thus the only remaining case is that $x$ and $y$ are both secondary and
  at distance two. We now first expose the entire embedding of $F-\{x,y\}$.
  Then either the image of $\neighbor(x)$ forms an $\alpha$-bad tuple, or
  it does not. In the former case $x$ is not mapped at all. In the latter
  case, $x$ is chosen uniformly among the at least $(1-\alpha)d^\Delta |V|$
  vertices in $U_x=\neighbor_G\big(h(\neighbor_F(x))\big)$. Likewise, we
  have that $y$ is either not mapped, or it is mapped to a vertex selected
  uniformly in a set $U_y$ with $|U_y|\ge(1-\alpha)d^\Delta |V|$. Hence (even though the sets~$U_x$ are and~$U_y$ are not independent), we get
  $\Prob[h(x)=u \text{ and } h(y)=v]\le
  (\frac{1}{(1-\alpha)d^\Delta|V|})^2\le(\frac{2}{d})^{4\Delta^2}\frac{1}{|V|^2}$.
\end{proof}

\begin{lemma}\label{lem:VC}
  Suppose that we are given $\alpha\in(0,\frac14)$, a forest $F$ of maximum
  degree at most $\Delta$ with a bipartition into primary and secondary
  vertices, and an $(\alpha,\Delta)$-superquasirandom graph $G=(V,E)$ of
  density $d$.

  Let $h$ be the limping homomorphism of $F$ to $G$. Suppose that $v\in V$
  is arbitrary, $x\in
  V(F)$ is an arbitrary primary vertex, and $y\in V(F)$ is an arbitrary
  secondary vertex. Then we have:
\begin{enumerate}[label=\abc]
\item\label{it:lemVCa} $\Prob\left[\exists z\in V(F)\setminus \{x\}\::\:
    h(x)=h(z)\right]\le \frac {v(F)}{|V|}$ and \\
  $\Prob\left[\exists z\in V(F)\setminus \{y\}\::\: h(x)=h(z) \;|\;h(y)=v\right]
  \le
  \frac{2v(F)}{d^\Delta(1-\alpha(\frac{2}{d})^\Delta)^{\Delta+3}|V|}$. 
\item\label{it:lemVCb} $\Prob\left[\exists z\in V(F)\setminus \{y\}\::\:
    h(y)=h(z)\right]\le \frac{2v(F)}{d^\Delta |V|}$ and \\
  $\Prob\left[\exists z\in V(F)\setminus \{y\}\::\: h(y)=h(z) \;|\;h(x)=v\right]
  \le \frac{2v(F)}{d^\Delta |V|}$. 
\item\label{it:lemVCc} For the number of colliding vertices $\VC=\{z\in V(F):\exists z'\::\:
  h(z)=h(z')\}$ and every $t>0$ we have $\Prob\left[|\VC|\ge
    \frac{2v(F)^2}{d^\Delta |V|}+t\right]\le
  2\exp(-\frac{t^2}{2(\Delta+1)^2 v(F)})$.
\end{enumerate}
\end{lemma}
\begin{proof}
  \ref{it:lemVCa} We expose the entire embedding of
  $F-\big(\{x\}\cup \neighbor_F(x)\big)$. 
	This is compatible with the order of embedding in the definition of limping homomorphisms
	because all vertices in~$\neighbor_F(x)$ are secondary, and they are the only
  secondary vertices whose embedding depends on the embedding of~$x$.
  Let~$W$ be the image of the vertices in $F-\big(\{x\}\cup
  \neighbor_F(x)\big)$. Observe that the event~$\mathcal E$ that there
  is~$z\in V(F)\setminus \{x\}$ with $h(x)=h(z)$ occurs if and only if the
  event~$\mathcal E'$ that $h(x)\in W$ occurs. But, no matter which
  vertices ended up in the set~$W$, the probability of~$\mathcal E'$
  (conditioned on~$W$) is $\frac{|W|}{|V|}\le \frac{v(F)}{|V|}$.  Hence
  $\Prob\left[\mathcal E\right]\le \frac {v(F)}{|V|}$.

  The second part of~\ref{it:lemVCa} follows from
  \begin{equation*}
    \Prob\left[\mathcal E \;|\;h(y)=v\right]
    =\frac{
      \Prob\left[h(y)=v \;|\; \mathcal E\right]
      \cdot \Prob\left[\mathcal E\right]
    }{\Prob[h(y)=v]}
    \le \frac{\frac{2}{d^\Delta|V|} \cdot \frac{v(F)}{|V|}}
    {\frac{\big(1-\alpha\left(\frac{2}{d}\right)^\Delta\big)^{\Delta+3}}{|V|}}
    \,,
  \end{equation*}
  where we use Lemma~\ref{lem:probh(x)=v}\ref{it:prob-sec} and
  Lemma~\ref{lem:probh(x)=v}\ref{it:prob-history}.

  \smallskip

  \noindent
  \ref{it:lemVCb} We expose the entire embedding of $F-\{y\}$. Let $W$ be
  the image of the vertices in $F-\{y\}$. Then we either know that~$y$ is
  skipped, or we place~$y$ u.a.r.\ in a set of size at least
  $(1-\alpha)d^\Delta|V|$. Similarly as in~\ref{it:lemVCa} the event we are
  interested in occurs if and only if $h(y)\in W$, which (conditioned
  on~$W$) has probability at most $\frac{|W|}{(1-\alpha)d^\Delta |V|}\le
  \frac{2v(F)}{d^\Delta |V|}$.
	This reasoning is valid even in the conditional space $h(x)=v$.

  \smallskip

  \noindent
  \ref{it:lemVCc} Using the bounds from~\ref{it:lemVCa}
  and~\ref{it:lemVCb}, we get $\Exp\left[|\VC|\right]\le
  \frac{2v(F)^2}{d^\Delta |V|}$.  We would now like to apply McDiarmid's
  inequality, Lemma~\ref{lem:McDiarmid}, to show concentration of
  $|\VC|$. For this purpose we consider the product space~$\Omega_F$
  from~\eqref{eq:productspace} and view $|\VC|$ as a function from~$\Omega_F$
  to~$\mathbb R$. We claim that $|\VC|$ is $2(\Delta+1)$-Lipschitz. Indeed,
  consider first the case that for a single secondary vertex~$y$ the random
  real~$\tau(y)$ changes. This only
  effects the embedding of~$y$ and hence $|\VC|$ changes by~$2$ at most. If,
  on the other hand, for a single primary vertex~$x$ the random choice
  of~$h(x)$ changes, then only the embedding of~$x$ and possibly its
  neighbours is effected. Hence in this case $|\VC|$ changes by at most
  $2(\Delta+1)$, as claimed. Therefore McDiarmid's Inequality
  (Lemma~\ref{lem:McDiarmid}) implies that
  \begin{equation*}
    \Prob\left[|\VC|\ge  \tfrac{2v(F)^2}{d^\Delta |V|}+t\right]
    \le \Prob\Big[|\VC|\ge \Exp\big[|\VC|\big]+t\Big]
    \le 2\exp\left(-\frac{2t^2}{(2(\Delta+1))^2v(F)}\right)\;.
    \qedhere
  \end{equation*}
\end{proof}

\section{Proof of the Nibble Lemma (Lemma~\ref{lem:newnibble})}\label{sec:proofofnibble}
Suppose that the numbers $\epsilon,\beta, c,\Delta$ are given. Let us take
$$0<\alpha\ll\alphaA\ll\alphaB\ll\alphaC\ll\alphaD\ll\alphaE\ll\beta\;.$$
That is we fix (in this order) $\alphaE$, $\alphaD$, $\alphaC$, $\alphaB$,
$\alphaA$, and $\alpha$ sufficiently small as a function of
$\epsilon,\beta, c,\Delta$, and of the previously fixed constants. Given
$r$, let $n_0$ be sufficiently large. Let~$\mathcal F$, $G$ and~$\mathcal
U_i$ be as in the setting of Lemma~\ref{lem:newnibble}.

%
%
%

For each $i\in [c]$ and each $s\in [k_i]$, 
the graph $G[V_{i,s}]$ has order at least $\varepsilon n$,
and hence, by Lemma~\ref{prop:SubgraphOfQuasiRandomGraph},
it is a $(3\alpha/\varepsilon^2)$-quasirandom graph of density $d\pm 3\alpha/\varepsilon^2$.
By Lemma~\ref{lem:superquasirandom}, this implies that $G[V_{i,s}]$ contains an almost spanning induced subgraph $G_{i,s}$ that is $(\alphaA,\Delta)$-superquasirandom and has order \marginpar{$m_{i,s}$}
\begin{equation}
\label{eq:lowerboundMIS}
m_{i,s}\ge (1- \alphaA)|V_{i,s}|>\epsilon n/2 
\end{equation}
and density $d_{i,s}=d\pm \alphaA$. 
Since $\mathcal U_i$ is $(\alpha,n)$-homogeneous, we have that $||U_{i,s}|-|U_{i,s'}||\le \alpha n$ for each $s,s'\in [k_i]$. Consequently, $m_{i,s}=(1\pm 2\alphaA) m_{i,s'}$. Thus, we can choose numbers $m_i>\epsilon n/2$ such that
\begin{equation}
\label{eq:Vis_samesize}
m_{i,s} = (1\pm \alphaA)  m_i.
\end{equation}
Finally, we recall that 
\begin{equation}
\label{eq:BoundsOnVFis}
(1-\alpha) \frac{n}{2r} \le n_{i,s}=v(F_{i,s}) \le \frac{2n}{r} .
\end{equation}

We now define the limping homomorphism~$h_{i,s}$ of~$(F_{i,s}-X_{i,s})$ to~$G_{i,s}$ so that the vertices of $V(F_{i,s})\setminus X_{i,s}$ of odd distance from~$X_{i,s}$ are the primary vertices and the ones at even distance are the secondary vertices. 
We denote the set of the primary and the secondary vertices in~$F_{i,s}$ by~\marginpar{$\primary_{i,s}$}$\primary_{i,s}$, 
and by~\marginpar{$\secondary_{i,s}$}$\secondary_{i,s}$, respectively. 
Let \marginpar{$\primary$}$\primary=\bigcup_{i,s}\primary_{i,s}$ 
and 
\marginpar{$\secondary$}$\secondary=\bigcup_{i,s}\secondary_{i,s}$.  
Let~\marginpar{$Y_{i,s}$}$Y_{i,s}$ denote the set of vertices skipped by~$h_{i,s}$.
Notice that 
$$
X_{i,s} \cap Y_{i,s} = \emptyset \qquand
F_{i,s}[X_{i,s} \cup Y_{i,s}] \text{ is an independent set\;,}
$$
because the vertices in $Y_{i,s}$ are at even distance from $X_{i,s}$ and hence in the same colour class as $X_{i,s}$.

Let $h:\bigcup_{i,s}F_{i,s}\to G$\marginpar{$h$} be the union of the homomorphisms $h_{i,s}$, and let~\marginpar{$H$}$H\subset G$ denote the image of the edges of the graphs $F_{i,s}$
under $h$, i.e. $H=\bigcup_{i,s}E(h_{i,s})$.

It is our goal to show that the random partial homomorphisms $h_{i,s}$ satisfy the assertions of the lemma with positive probability. We will show that each of the assertions is actually met with high probability. The following table shows lemmas corresponding to individual assertions:
\begin{center}
\begin{tabular}{cccccccc}
\hline 
\ref{cond:NumberSkippedVertices}  & \ref{cond:VertexCollisions}  & \ref{cond:EdgeCollisions}  & \ref{cond:NCN} & \ref{cond:NumberSkippedNeighbours} & \ref{cond:rootneighbours} & \ref{cond:GStillQuasirandom} & \ref{cond:UStillSmallLoad} \\ 
Lem~\ref{lem:Yis_small} & Lem~\ref{lem:vertexcollisions} & Lem~\ref{lem:edgecollisions} & Lem~\ref{lem:NCN} & Lem~\ref{lem:skippedInneighbourhood} & Lem~\ref{lem:XN} & Lem~\ref{lem:quasirandom} & Lem~\ref{lem:load} and Lem~\ref{lem:homogeneous} \\ 
\hline 
\end{tabular} 
\end{center}

In addition to the parameters controlled by the lemma, we need to control the following quantities.
For $v\in V$, define $D_P(v)$ and $D_S(v)$ to be the number of all 
primary and secondary vertices, respectively,
that are mapped to $v$,\marginpar{$D_P(v)$}\marginpar{$D_S(v)$}
\begin{align*}
D_P(v) = \left| h^{-1}(v) \cap \primary \right| \quad \mbox{and} \quad
D_S(v) = \left| h^{-1}(v) \cap \secondary \right| .
%
\end{align*}

\begin{lemma}
\label{lem:DPandDSaresmall} We have
\begin{align}
\label{eq:DPand1}
\Prob\left[ \exists v\in V \colon D_P(v) > \frac{15n}{\eps r} \right] &\le \exp(-\sqrt{n})\;\mbox{, and}\\
\label{eq:DPand2}
\Prob\left[ \exists v\in V \colon D_S(v) > \frac{15n}{\eps r}
\right] 
&\le \exp(-\sqrt{n})\;.
\end{align}
Further, the same bounds hold, if we condition on $h(z)=u$ for an arbitrary
$z\in V(F_{i,s})$ with $i\in[c]$ and $s\in[k_i]$ and $u\in
V(G_{i,s})$. 
\end{lemma}
\begin{proof}
We fix a vertex $v\in V$ and first compute the expected number of primary vertices mapped to $v$. 
For every $i\in [c]$ and $s\in[k_i]$, we embed at most
$v(F_{i,s})\le  2n/r$ primary vertices into the set~$V(G_{i,s})$ with 
$m_{i,s} \ge \eps n/2$ vertices. Since there are at most $2n$ choices of pairs $(i,s)$, this gives that
$\Exp[D_P(v)]=\sum_{i,s}\sum_{v\in V(F_{i,s})}\frac{1}{m_{i,s}}\le  \frac{8n}{\eps r}$.
The Chernoff bound~\eqref{eq:Ch2} with $\mu=8n/(\eps r)$ and $\delta=\frac12$ and a union bound over all choices of $v$ gives~\eqref{eq:DPand1}.

\smallskip

To prove~\eqref{eq:DPand2}, let us again fix a vertex $v\in V$. 
Lemma~\ref{lem:probh(x)=v}\ref{it:prob-sec} gives that for a fixed secondary vertex~$y$,
\begin{equation}
\label{eq:ProbVisHit}
\Prob[h(y)=v]
\le \frac{\left(1+\alphaA\left(\frac{2}{d\pm 2\alphaA}\right)^{\Delta}\right)^{\Delta+3}}{m_{i,s}}
\leByRef{eq:lowerboundMIS} \frac{3}{\epsilon n}.
\end{equation}

For each $(i,s)$ consider the square $F_{i,s}^2[\secondary_{i,s}]$ of the graph $F_{i,s}[\secondary_{i,s}]$. 
This graph has maximum degree at most $\Delta^2$, and thus is $(\Delta^2+1)$-colourable.
Let $V(F_{i,s})=C_{i,s}^1\dcup \dots\dcup C_{i,s}^{\Delta^2+1}$ be a colouring of $F^2_{i,s}[\secondary_{i,s}]$.
Note that the events $h(x)=v$ and $h(x')=v$ for $x\neq x' \in C_{i,s}^{\ell}$ are independent,
because the unique $x,x'$-path in $F_{i,s}$ contains at least two primary vertices.
The same reasoning gives that the events $\{h(x)=v\}_{x\in C_{i,s}^{\ell}}$
are in fact mutually independent.

We let $C^{\ell}= \bigcup_{i,s} C_{i,s}^{\ell}$ and $Z^\ell=\left|C^\ell\cap h^{-1}(v)\right|$.
Since we have at most $2n$ forests $F_{i,s}$, it follows from~\eqref{eq:BoundsOnVFis} that
\begin{equation}
\label{eq:upperCL}
|C^{\ell}| \le \sum_{\ell'} |C^{\ell'}| \le \sum_{i,s} v(F_{i,s}) \le 4n^2 / r \,.
\end{equation}
Thanks to the bound in~\eqref{eq:ProbVisHit} and the mutual independence described above, the random variable
$Z^{\ell}$ is stochastically dominated by a random variable $Z\in\mathrm{Bin}(|C^{\ell}|,3/(\eps n))$. 
We would like to apply the Chernoff bound in~\eqref{eq:Ch3} with
\[
\mu=|C^{\ell}|3/(\eps n) \text{\quad and \quad}
\delta'=1 + \frac{1}{10^3(\Delta^2+1)} \text{\quad and \quad} 
t= \mu + \frac{n}{10 (\Delta^2+1)\eps r} \,.
\]
We check the condition of~\eqref{eq:Ch3},
\begin{equation*}
\delta' \mu = \left( 1+\frac{1}{10^3 (\Delta^2+1)}\right) \mu = \mu + \frac{3|C^{\ell}|}{10^3 (\Delta^2 +1)\eps n}
\leByRef{eq:upperCL} \mu + \frac{12n}{10^3 (\Delta^2 +1)\eps r} \le t.
\end{equation*}
%
Hence we can indeed apply~\eqref{eq:Ch3} and obtain $\delta''>0$ 
(independent of $n$) for which 
\[
\Prob\left[Z^{\ell}\ge \mu + \frac{n}{10 (\Delta^2+1)\eps r}\right] \le \exp\left(-\delta'' \frac{n}{10 (\Delta^2+1)\eps r}\right) \,.
\]
By a union bound over all $\ell\in [\Delta^2+1]$ we get that with probability at least $1-\exp(-n^{2/3})$
\begin{align*}
D_S(v) &= \sum_{\ell=1}^{\Delta^2 +1} Z^\ell 
\le \sum_{\ell=1}^{\Delta^2 +1} \left(\mu + \frac{n}{10 (\Delta^2+1)\eps r}\right)\\ 
&= \sum_{\ell=1}^{\Delta^2 +1} \left(|C^{\ell}| 3/(\eps n) + \frac{n}{10 (\Delta^2+1)\eps r}\right)
\leByRef{eq:upperCL} \frac{3}{\eps n} \frac{4n^2}{r} + \frac{1}{10} \frac{n(\Delta^2+1)}{(\Delta^2+1)\eps r}
\le \frac{14n}{\eps r}.
\end{align*}
Finally, another union bound over all $v\in V$ 
shows that~\eqref{eq:DPand2} is satisfied.

Since the placement of all but at most $\Delta^2+1$ of the forest vertices
is independent of the placement of~$z$ we also get the bounds
from~\eqref{eq:DPand1} and~\eqref{eq:DPand2} if we condition on $h(z)=u$.
\end{proof}

\begin{lemma}
\label{lem:ImageDegree}
Let $z\in V(F_{i,s})$ with $i\in[c]$ and $s\in[k_i]$ and $v\in
V(G_{i,s})$ be arbitrary. 
\[\Prob\left[\Delta(H)>\frac{30\Delta n}{\epsilon r}\right]\le
2\exp(-\sqrt{n})
\quad\text{and}\quad
\Prob\left[\Delta(H)>\frac{30\Delta n}{\epsilon r} \;|\; h(z)=v\right]\le
2\exp(-\sqrt{n})
\;.\]
\end{lemma}
\begin{proof}
This follows from the fact that 
$\Delta(H)\le \Delta\cdot\max_v(D_P(v)+D_S(v))$
and from Lemma~\ref{lem:DPandDSaresmall}.
\end{proof}

\begin{lemma}
\label{lem:Yis_small} 
We have
$$
\Prob\left[\forall i\in [c] \quad \forall s\in [k_i]\colon \quad |Y_{i,s}| \le \frac{\beta n}{r}\right] 
\ge 1- \exp(-\sqrt{n}).
$$
\end{lemma}
\begin{proof}
Fix $i\in [c]$ and $s\in [k_i]$. 
By Lemma~\ref{lem:probh(x)=v}\ref{it:prob-skipped}, for the number of vertices skipped by $h_{i,s}$ we have 
$\Exp\left[|Y_{i,s}|\right]\le \alphaA \frac {2n}{r}$. 
Note that the number of skipped vertices is $\Delta$-Lipschitz. McDiarmid's Inequality (Lemma~\ref{lem:McDiarmid}) 
with $t=\alphaA 2n/r$ and $k= v(F_{i,s}) \le 2n/r$ gives that
$\Prob\left[|Y_{i,s}|>2\cdot \frac {\alphaA 2n}{r}\right]\le 2\cdot \exp(\frac {-8 \alphaA^2n^2r}{r^22n\Delta^2})=
2\exp(-\frac{4\alphaA^2 n}{r\Delta^2})$. Hence using the union bound over all choices of
$(i,s)$ we obtain
\[\Prob\left[\exists i,s\colon \quad |Y_{i,s}| > \frac{\beta n}{r}\right] \quad\le\quad
\Prob\left[\exists i,s\colon \quad |Y_{i,s}| > \frac {4\alphaA n}{r}\right] \quad\le\quad
\exp(-\sqrt{n})\;.     \qedhere\]
\end{proof}

\begin{lemma}
\label{lem:vertexcollisions}
We have
\[
\Prob\left[\forall i\in [c] \quad \forall s\in [k_i]\colon \quad 
|\VC_{i,s}|\le \frac{20n}{\epsilon r^2d^\Delta}\right] 
\ge 1- \exp(-\sqrt{n}) \;.
\]
\end{lemma}
\begin{proof}
Fix $i\in [c]$ and $s\in [k_i]$. We first observe that
\[
\frac{2 v(F_{i,s})^2}{d_{i,s}^\Delta m_{i,s}} + \frac{n}{\eps r^2 d^\Delta}
\quad\le\quad \frac{4 (2n/r)^2}{\frac{9}{10}d^\Delta\eps n} + \frac{n}{\eps r^2 d^\Delta}
\quad\le\quad \frac{20n}{\eps r^2 d^\Delta} \;.
\]
Hence Lemma~\ref{lem:VC}\ref{it:lemVCc} with $t=\frac{n}{\eps r^2 d^\Delta}$ gives that  
\begin{equation*}\begin{split}
\Prob\left[|\VC_{i,s}|\ge \frac{20n}{\epsilon r^2d^\Delta}\right]
&\le \Prob\left[|\VC_{i,s}|\ge\frac{2v(F_{i,s})^2}{d_{i,s}^\Delta m_{i,s}}+\frac{n}{\epsilon r^2d^\Delta}\right]\\
&\le 2\exp\left(-\frac {n^2}{2\epsilon^2r^4d^{2\Delta}(\Delta+1)^2n_{i,s}}\right)
\le 2\exp\left(-\frac{n}{4\eps^2 r^3 d^{2\Delta}(\Delta+1)^2}\right)\;.
\end{split}\end{equation*}
Using a union bound over all choices $(i,s)$, we get the statement of the lemma.
\end{proof}

Recall that $\EC_{i,s}$ contains all the vertices of~$F_{i,s}$ that are
contained in an edge collision. 
We define \marginpar{$\EC^*_{i,s}$}$\EC^*_{i,s}=\{xy\in E(F_{i,s})\colon xy \text{ is colliding}\}$.
Notice that $|\EC_{i,s}|\le2|\EC^*_{i,s}|$.

\begin{lemma}
\label{lem:colliding-edge}
Let $xy\in E(F_{i,s})$ be an edge with $x\in\primary_{i,s}$ and
$y\in\secondary_{i,s}$. 
Let $z\in V(F_{i,s})\setminus\{y\}$ and $v\in V(G_{i,s})$.
Then we have
\begin{equation*}
  \Prob\left[xy \in \EC^*_{i,s}\;|\; h(z)=v\right]\le \frac{61\Delta }{\epsilon^2rd^{\Delta}}\;
  \quad\text{and}\quad
  \Prob\left[y \in \EC_{i,s} \;|\; h(z)=v\right]\le \frac{61\Delta^2}{\epsilon^2 rd^{\Delta}}\;,
\end{equation*}
and hence also $\Prob\left[xy \in \EC^*_{i,s}\right]\le\frac{61\Delta}{\epsilon^2 rd^{\Delta}}$.
\end{lemma}

\begin{proof}
Let $u= h(x)$ and~$z$ be an arbitrary vertex
in $F_{i,s}-y$. Let
$\{u_1,\ldots,u_p\}=h(\neighbor_{F_{i,s}}(y)\setminus \{x\})$.
We denote by $\mathcal{B}$ the event that $\{ u,u_1,\ldots,u_p\}$ forms an $\alphaA$-bad set. 
First observe that $p_1=\Prob[xy\in\EC^*_{i,s} | h(z)=v \text{ and } \mathcal{B} ] =0$, because the
event $\mathcal{B}$ implies that $y$ is skipped and thus the edge $xy$
is not colliding. On the other hand, if $\mathcal{B}$ does not occur, then 
\begin{equation}
\label{eq:UsetNotBad}
\left|\neighbor_{G_{i,s}}(u,u_1,\dots,u_p)\right| \ge (1-\alphaA)(d-\alphaA)^\Delta m_{i,s} \ge \frac12 d^\Delta \eps n .
\end{equation}
Next we define 
\[
\tilde \neighbor(u)=\left\{w\in \neighbor_G(u)\:\colon \exists i\in [c]\;\exists s\in [k_i]\;\exists x'y'\in E(F_{i,s}) \text{ with } xy\neq x'y' \text{ and } h(x'y')=uw \right\}\;.
\]
This means that the edge~$xy$ is colliding only if~$y$ is mapped to $\tilde \neighbor(u)$.
%
%
By Lemma~\ref{lem:ImageDegree} we have 
\[
p_2=\Prob\Big[|\tilde \neighbor(u)|> \frac{30\Delta n}{\epsilon r} \;\Big|\; h(z)=v\Big] \le 2\exp(-\sqrt{n}) .
\]
Moreover, because $z\neq y$ we have
\begin{align*}
p_3=\Prob\left[xy \in \EC^*_{i,s}\;\Big|\; 
h(z)=v \text{ and } \mathcal{\overline{B}}\text{ and } |\tilde \neighbor(u)|\le \frac{30\Delta n}{\epsilon r}\right]
&\leByRef{eq:UsetNotBad} \frac{30\Delta n}{\epsilon r\cdot d^\Delta\epsilon n/2}= \frac{60\Delta }{\epsilon^2 rd^{\Delta}}\;.
\end{align*}
Since $\Prob[xy\in \EC^*_{i,s} \;|\; h(z)=v]\le p_1+p_2+p_3$,
we obtain that $\Prob[xy\in \EC^*_{i,s} \;|\;h(z)=v]
\le \frac{61\Delta }{\epsilon^2 rd^{\Delta}}$.
In addition,
\begin{equation*}
\Prob\left[y \in \EC_{i,s}\;|\;h(z)=v\right]
\le\sum_{x\in \neighbor_{F_{i,s}}(y)}\Prob[xy\in \EC^*_{i,s} \;|\;h(z)=v]
\le\Delta\frac{61\Delta }{\epsilon^2 rd^{\Delta}}\;.
\qedhere
\end{equation*}
\end{proof}

\begin{lemma}
\label{lem:edgecollisions}
We have
\[\Prob\left[\exists i\in [c], s\in [k_i]\colon |\EC_{i,s}|>\frac{300 \Delta n}{\epsilon^2r^2d^{\Delta}} \right]\le
\exp(-\sqrt{n})\;.
\]
\end{lemma}
\begin{proof}
Fix an arbitrary $i\in [c]$ and an arbitrary $s\in [k_i]$. Combining Lemma~\ref{lem:colliding-edge} and~\eqref{eq:BoundsOnVFis}, we get $\Exp[|\EC_{i,s}|]\le 2\Exp[|\EC^*_{i,s}|]\le 
\frac{244\Delta n}{\epsilon^2 r^2d^{\Delta}}$. 

Changing the value at a single coordinate in~\eqref{eq:productspace} leads to a change of the placement of at most $\Delta^2$ edges. A change of a placement of a single edge can change the number of colliding edges by at most~2, which can result in a change of at most~4 in $|\EC_{i,s}|$. We conclude that $|\EC_{i,s}|$ is $4\Delta^2$-Lipschitz. 

McDiarmid's Inequality (Lemma~\ref{lem:McDiarmid}) gives that
$$\Prob\left[|\EC_{i,s}|\ge \frac{300\Delta n}{\epsilon^2 r^2d^{\Delta}}\right]\le 2\exp\left(-\frac{2(\tfrac{56\Delta n}{\epsilon^2 r^2d^{\Delta}})^2}{16\Delta^4\cdot\frac{2n}{r}}\right)\le\exp(-n^{0.9})\;.$$
The lemma then follows by a union bound over $i$ and $s$.
\end{proof}

\begin{lemma}
\label{lem:NCN}
We have
\[\Prob\left[\exists v\in V\colon |\FN(v)|>\frac
{10^4\Delta^3n}{\epsilon^3r^2d^{2\Delta}}\right]\le \exp(-\sqrt{n})\;.
\]
\end{lemma}
\begin{proof}
Fix a vertex $v\in V$. Fix $i\in [c]$ and $s\in [k_i]$.
\begin{claim}\label{cl:bad}
  Let $xy\in E(F_{i,s})$. Then $\Prob[h(x)=v \text{ and } y\text{ is faulty}]\le
\frac{10^3\Delta^2}{\epsilon^3rd^{2\Delta}n}$.
\end{claim}
\begin{claimproof}[Proof of Claim~\ref{cl:bad}] 
We shall use
\begin{align}\label{eq:bad:cond}
  \Prob[h(x)=v \text{ and } y\text{ is faulty}]=\Prob[h(x)=v]\cdot \Prob[y\text{ is faulty}\:|\: h(x)=v]\;.
\end{align}
First consider the case that~$x$ is primary.  The secondary vertex~$y$ is
faulty if it is colliding, or if it is in an edge collision (due to either
$xy$ colliding, or $yz$ colliding with $z\in
\neighbor_{F_{i,s}}(y)\setminus \{x\}$). 
For vertex collisions, by Lemma~\ref{lem:VC}\ref{it:lemVCb},
we have
\begin{equation*} 
  \Prob[y\in \VC_{i,s}\:|\: h(x)=v] 
  \le \frac{2n_{i,s}}{d^\Delta m_{i,s}}
  \leBy{\eqref{eq:lowerboundMIS},\eqref{eq:BoundsOnVFis}}\frac{8}{ \epsilon rd^\Delta}\;.
\end{equation*}
For edge collisions, Lemma~\ref{lem:colliding-edge} gives 
$\Prob[y\in \EC_{i,s}\:|\: h(x)=v]\le \frac{61\Delta^2}{\epsilon^2rd^{\Delta}}$.
Hence 
\begin{equation*}
    \Prob[y\text{ is faulty}\:|\: h(x)=v]
    \le  \frac{8}{ \epsilon rd^\Delta} + \frac{61\Delta^2}{\epsilon^2rd^{\Delta}}\
    \le \frac{70\Delta^2}{\epsilon^2rd^{\Delta}}\;.
\end{equation*}
Since $\Prob[h(x)=v]=\frac {1}{m_{i,s}}\le \frac 2{\epsilon n}$
by Lemma~\ref{lem:probh(x)=v}\ref{it:prob-prim}, we get together
with~\eqref{eq:bad:cond} that
\begin{equation*}
  \Prob[h(x)=v \text{ and } y\text{ is faulty}] \le \frac{140\Delta^2}{\epsilon^3rd^{\Delta}n}\;,
\end{equation*}
which gives the claim in this case.

\smallskip 
Next, consider the case that~$x$ is secondary.  We denote by
$\mathcal A_1$ the event that $y$ is in a vertex collision. We denote by
$\mathcal A_2$ the event that $y$ together with some vertex
$z\in\neighbor(y)\setminus \{x\}$ forms a colliding edge. We denote by
$\mathcal A_3$ the event that $xy$ is colliding. The primary vertex $y$ is
faulty if at least one of the events $\mathcal A_1$, $\mathcal A_2$, or
$\mathcal A_3$ occurs.

By Lemma~\ref{lem:VC}\ref{it:lemVCa} we have
\begin{equation}\label{eq:A1}
\Prob[\mathcal A_1\:|\:h(x)=v]
\le\frac{2n_{i,s}}{d^\Delta\frac12m_{i,s}}
\le\frac{16}{rd^\Delta\eps}\,.
\end{equation}
For $z\in\neighbor(y)\setminus\{x\}$ fixed, Lemma~\ref{lem:colliding-edge} gives
$\Prob[yz\in \EC_{i,s}\:|\: h(x)=v]\le
\frac{61\Delta}{\epsilon^2rd^{\Delta}}$. Hence we obtain
\begin{equation}\label{eq:A2}
\Prob[\mathcal A_2\:|\:h(x)=v]
\le\Delta\frac{61\Delta}{\epsilon^2rd^{\Delta}}
\,.
\end{equation}
In order to obtain a similar bound for the event~$\mathcal A_3$, let
$\mathcal A'_3$ be the event that $xy$ is in an edge collision with an edge
from a different forest $F_{i',s'}$. Let~$H'$ be the graph formed by the
images of all forests but $F_{i,s}$, that is, $H-E(h_{i,s})$. Now fix a
mapping of all forests but~$F_{i,s}$. By Lemma~\ref{lem:ImageDegree},
with probability at least $1-2\exp(-\sqrt n)$ we have
$\Delta(H')\le\frac{30\Delta n}{\eps r}$ (and this is independent of the
event $h(x)=v$). 
Assume that this is the case and let $\Prob_{i,s,H'}$ be (the measure
on) the conditional
probability space associated with the limping homomorphism for~$F_{i,s}$.
We have,
\begin{equation*}\begin{split}
  \Prob_{i,s,H'}[\mathcal A'_3\;|\;h(x)=v] &=
  \Prob_{i,s,H'}[h(y)\in\neighbor_{H'}(v)\;|\;h(x)=v] \\
  &\le\sum_{u\in\neighbor_{H'}(v)} \Prob_{i,s,H'}[h(y)=u\;|\;h(x)=v]\,.
\end{split}\end{equation*}
For a fixed vertex $u\in V$ we have
\[\Prob_{i,s,H'}[h(y)=u\;|\;h(x)=v]=\Prob_{i,s,H'}[h(y)=u\text{
and }h(x)=v]/\Prob_{i,s,H'}[h(x)=v]\,,\] which by
Lemma~\ref{lem:probh(x)=v}\ref{it:probedge}
and~Lemma~\ref{lem:probh(x)=v}\ref{it:prob-sec} is at most
$(\frac2{dm^2_{i,s}})/\frac{1}{2m_{i,s}}=\frac4{dm_{i,s}}$. Hence,
\[\Prob_{i,s,H'}[\mathcal A'_3\;|\;h(x)=v]\le\frac{30\Delta n}{\eps
  r}\cdot\frac4{dm_{i,s}}\le\frac{120\Delta}{\eps^2dr}\,.\] Returning to our
original probability space we thus obtain $\Prob[\mathcal
A'_3\;|\;h(x)=v]\le 2\exp(-\sqrt n)+\frac{120\Delta}{\eps^2dr}$.
Since
\[\Prob[\mathcal A_3\;|\; h(x)=v]\le\Prob[\mathcal A'_3\;|\;
h(x)=v]+\Prob[y\in \VC_{i,s}\;|\; h(x)=v]\;,\] we conclude from
Lemma~\ref{lem:VC}\ref{it:lemVCa} 
that 
\begin{equation}\label{eq:A3}
  \Prob[\mathcal A_3\;|\; h(x)=v]
  \le \frac{121\Delta}{\eps^2dr} + \frac{2n_{i,s}}{d^\Delta\frac12m_{i,s}}
  \le \frac{121\Delta}{\eps^2dr} + \frac{16}{rd^\Delta\eps}
  \le \frac{150\Delta}{\eps^2rd^{\Delta}}\,.
\end{equation}

Finally, since $\Prob[h(x)=v]\le\frac {2}{m_{i,s}}\le \frac 4{\epsilon n}$
by Lemma~\ref{lem:probh(x)=v}\ref{it:prob-sec}, we get
from~\eqref{eq:bad:cond}, \eqref{eq:A1}, \eqref{eq:A2} and~\eqref{eq:A3}
that
\begin{equation*}
  \Prob[h(x)=v \text{ and } y\text{ is faulty}] \le 
  \frac{4}{\eps n} \Big( 
  \frac{16}{\eps rd^\Delta}
  +\frac{61\Delta^2}{\epsilon^2rd^{\Delta}}
  +\frac{150\Delta}{\eps^2rd^\Delta}
  \Big)
  \le \frac{4}{\eps n} \cdot \frac{200\Delta^2}{\eps^2 r d^\Delta}
  \;,
\end{equation*}
which also gives the claim in this case.
\end{claimproof}

For $x\in V(F_{i,s})$ let $\mathcal E_{x,v}$ be the event that $h(x)=v$
and that there exists a vertex $y\in \neighbor_{F_{i,s}}(x)$ such that $y$ is faulty.

\begin{claim}\label{cl:bad:vertex}
  For each $x\in V(F_{i,s})$ we have 
  $\Prob[\mathcal E_{x,v}]\le \frac{10^3\Delta^3}{\epsilon^3rd^{2\Delta}n}$.
\end{claim}
\begin{claimproof}[Proof of Claim~\ref{cl:bad:vertex}]
  This follows immediately from $\Delta(F_{i,s})\le\Delta$ and Claim~\ref{cl:bad}.
\end{claimproof}

We now return to the proof of Lemma~\ref{lem:NCN} and recall that
$\FN(v)=\bigcup_{i,s}\{x\in V(F_{i,s})\colon \mathcal E_{x,v}\}$. Thus, combining the above claim and e.g.~\eqref{eq:upperCL}, we get
\begin{equation}\label{eq:FNexp}
\Exp[|\FN(v)|]\le \frac{10^3\Delta^3}{\epsilon^3rd^{2\Delta}n}\cdot\frac{4n^2}r\;.
\end{equation}
Next, we argue that $|\FN(v)|$ is $4(\Delta+1)^2$-Lipschitz. To see this, we need to control the effects a change of a single variable in~\eqref{eq:productspace} may have \emph{(a)} on the number of vertices that are mapped to $v$, and \emph{(b)} on the number of vertices that are faulty.  Observe that a change of a single vertex being faulty or not may lead to a change up to~$\Delta$ in the value $|\FN(v)|$. 
\begin{enumerate}
\item[\emph{(a)}] A change of a single variable in~\eqref{eq:productspace} can alter the position of at most $\Delta+1$ vertices.
\item[\emph{(b)}] A change in the position of a single vertex can alter the total number of faulty vertices by at most~$4$. By \emph{(a)}, we get that a change of a single variable in~\eqref{eq:productspace} can alter the number of faulty vertices by at most $4(\Delta+1)$.
\end{enumerate}

Thus by \emph{(a)} and \emph{(b)} we get that $|\FN(v)|$ is $(\Delta+1)+\Delta\cdot (4(\Delta+1))$-Lipschitz.

Next, fix~$\Lambda\in \mathbb N$ and assume that $|\FN(v)|\ge \Lambda$ for a particular realization in~\eqref{eq:productspace}. We claim that there is a set of at most $\Lambda\cdot 3(\Delta+1)$ coordinates that certifies that $|\FN(v)|\ge \Lambda$. Indeed, each elementary contribution to $|\FN(v)|$ corresponds to some vertex~$x$ mapped to $v$ whose neighbour is faulty. 
To certify that a vertex is faulty, we need to encode its position (or the position of the colliding edge incident to this vertex) and the position of the vertex (edge) with which it collides. To encode the position of a secondary vertex~$y$, we need to know $\Delta+1$ coordinates from~\eqref{eq:productspace}. These coordinates also give the position of any primary vertex that may be incident to any colliding edge containing~$y$. So we need at most $2(\Delta+1)$ coordinates to certify that a secondary vertex is faulty. The number of coordinates needed to certify that a primary vertex is faulty is also bounded by $2(\Delta+1)$. 
So, to increase $|\FN(v)|$ by one, we need at most $(\Delta+1)$ coordinates
to certify the position of~$x$ and at most $2(\Delta+1)$ coordinates to certify that~$x$ has a faulty neighbour. The claim follows.

Therefore $|\FN(v)|$ satisfies all the conditions of Talagrand's Inequality (Lemma~\ref{lem:Talagrand}). We get
\begin{align*}
\Prob\left[|\FN(v)|>\frac{10^4\Delta^3n}{\varepsilon^3r^2d^{2\Delta}}\right]&\le
\Prob\left[|\FN(v)|>\Exp[|\FN(v)|]+\frac{10^4\Delta^3n}{2\varepsilon^3r^2d^{2\Delta}}\right]\le \exp(-\Theta (n))\;.
\end{align*}
 The lemma follows by a union bound over all choices of $v$.

\end{proof}

\begin{lemma}
\label{lem:skippedInneighbourhood}
We have
\[
\Prob\left[\exists v\in V\colon |\YN(v)|>\frac
{\beta n}{r}\right]\le \exp(-\sqrt{n})\;.
\]
\end{lemma}
\begin{proof}
  We proceed similarly as in the proof of the previous lemma.
  Fix $v\in V$. For $x\in V(F_{i,s})$ denote by $\mathcal E_{x,v}$ the
  event that $h(x)=v$ and that there exists a vertex in 
	$\neighbor_{F_{i,s}}(x)$ that is skipped.
	By Lemma~\ref{lem:probh(x)=v}\ref{it:prob-prim} and
  Lemma~\ref{lem:probh(x)=v}\ref{it:prob-skippedConditioned} we have
  \begin{equation}\label{eq:YN:Exv}
    \Prob[\mathcal E_{x,v}]
    \le\sum_{y\in \neighbor_{F_{i,s}}(x)}\Prob[y\text{ is skipped}\:|\:h(x)=v]\cdot\Prob[h(x)=v]
    \le\Delta\cdot\alphaA\cdot \frac 2{\eps n}\,.
  \end{equation}
  Observe that $\YN(v)=\bigcup_{i,s}\{x\in V(F_{i,s})\colon\mathcal
  E_{x,v}\}$. Moreover, for $x,x'\in \bigcup_{i,s}V(F_{i,s})$ of distance
  at least~$6$ the events $\mathcal E_{x,v}$ and $\mathcal E_{x',v}$ are
  independent. Therefore we consider the $6$-th power~$F^6$ of
  $\bigcup_{i,s}F_{i,s}$. Since~$F^6$ has maximum degree less than
  $\Delta^6$ this graph has a $\Delta^6$-colouring
  $\bigcup_{i,s}V(F_{i,s})=C^1\dcup\dots\dcup C^{\Delta^6}$. For $\ell\in\Delta^6$ let $Z^\ell$ be the number of $x\in
  C^\ell$ such that $\mathcal E_{x,v}$ holds. By~\eqref{eq:YN:Exv} the random
  variable $Z^\ell$ is stochastically dominated by
  $\mathrm{Bin}(|C^\ell|,\frac{2\Delta\alphaA}{\eps n})$. Thus we can apply
  Chernoff's inequality~\eqref{eq:Ch3} with
  \begin{equation*}
    \mu=\frac{2\Delta\alphaA}{\eps n}|C^\ell|\,,
    \quad
    \delta'=1+\frac{\eps}{8\Delta^7}
    \quad\text{and}\quad
    t=\mu+\frac{\alphaA n}{r\Delta^6}\,,
  \end{equation*}
  which is possible because
  $\delta'\mu\le\mu+\frac{\eps}{8\Delta^7}\cdot\frac{2\Delta\alphaA}{\eps
    n}\cdot\frac{4 n^2}{r}=t$. We conclude that there is $\delta''>0$ such
  that
  \begin{equation*}
    \Prob\Big[ Z^\ell\ge\frac{2\Delta\alphaA}{\eps
      n}|C^\ell|+\frac{\alphaA n}{r\Delta^6}\Big]
    \le \exp\Big(-\delta''\frac{\alphaA n}{r\Delta^6}\Big)\,.
  \end{equation*}
  Hence with probability at least $1-\Delta^6\cdot
  \exp\big(-\delta''\frac{\alphaA n}{r\Delta^6}\big)$ we have
  \begin{equation*}
    |\YN(v)|
    =\sum_{\ell\in[\Delta^6]} Z^\ell
    \le\frac{2\Delta\alphaA}{\eps n}\cdot\frac{4 n^2}{r}+\Delta^6\frac{\alphaA n}{r\Delta^6}
    \le9\frac{\Delta\alphaA n}{\eps r}<\frac{\beta n}{r}\,.
  \end{equation*}
  The lemma follows by a union bound over $v\in V$.
\end{proof}

\begin{lemma}\label{lem:XN} 
We have
  \[\Prob\left[\exists v\in V\colon|\XN(v)|> \frac{\beta n}{r}\right]\le\exp(-\sqrt{n})\,.\]
\end{lemma}
\begin{proof}
  Fix $v\in V$. For $i\in[c]$ and $s\in[k_i]$ let 
	$Q_{i,s}=\bigcup_{x\in X_{i,s}}\neighbor_{F_{i,s}}(x)$. 
	By definition each $y\in Q_{i,s}$ is
  primary, hence~$y$ gets mapped to~$v$ with probability at most
  $\frac{2}{\epsilon n}$ by Lemma~\ref{lem:probh(x)=v}\ref{it:prob-prim}.
  These events are independent, and thus the number of vertices in
  $\bigcup_{i,s} Q_{i,s}$ which are mapped to $v$ is stochastically
  dominated by $\mathrm{Bin}(\frac{2}{\epsilon n},\sum_{i,s}
  |Q_{i,s}|)$. We have $\sum_{i,s}|Q_{i,s}|\le \Delta
  \sum_{i,s}|X_{i,s}|\le \Delta \cdot 2n\frac{\alpha n}r$.  Thus, by
  Chernoff's inequality~\eqref{eq:Ch2} applied with $\mu=\frac{2}{\eps
    n}\cdot \frac{2\Delta\alpha n^2}{r}=\frac{4\Delta\alpha n}{\eps r}$ and
  $\delta=1$ we have
  \begin{equation*}
    \Prob\Big[|\XN(v)|>\frac{\beta n}{r}\Big]
    \le\Prob\Big[|\XN(v)|>\frac{8\Delta\alpha n}{\epsilon r}\Big]
    \le 2\exp\Big(-\frac{4\Delta\alpha n}{3\eps r}\Big)
    \;.
  \end{equation*}
The lemma follows by taking the union bound over all choices of $v$.
\end{proof}


We now prepare for the proof of~\ref{cond:GStillQuasirandom}.

\begin{definition}[important group]
\label{def:importantgroup}
We say that $i\in [c]$ is an \emph{important group}\marginpar{important~g.} if $k_i>\frac
{\sqrt{\alpha}nr}{2}$. The set of important groups is denoted by $\impgr\subset [c]$.\marginpar{$\impgr$}
\end{definition}

\begin{lemma}\label{lem:Nimportant}
The total number of edges in forests~$(F_{i,s})_{i,s}$ from non-important groups is less than 
$\beta n^2/16$.
\end{lemma}
\begin{proof}
By definition there are at most
$\frac{\sqrt{\alpha}nr}{2}$ forests in each non-important groups and each such forest has at most $\frac{2n}{r}$ edges. The number of non-important groups is at most $c$.  
As $c\sqrt{\alpha}<\beta/16$,  the claim follows.
\end{proof}

\begin{definition}[typical]\label{def:typical}
Let $i\in [c]$.
A pair $uv\in \binom{V}{2}$ is called \DEF{$i$-typical} if $(\load(u,v,\cU_i)-\mu(\cU_i))^2\le\sqrt{\alpha}n^2$, and \DEF{$i$-atypical}, otherwise.
An edge $uv\in E$ is called \DEF{typical}, if it is $i$-typical for each $i\in [c]$, and \DEF{atypical} otherwise.
\end{definition}

\begin{lemma}\label{lem:atypical}
  For each~$i\in [c]$ there are at most $\sqrt[4]{\alpha} n^2$ pairs in
  $\binom{V}{2}$ that are $i$-atypical. Consequently, there are at most
  $\beta n^2/16$ atypical edges in the graph~$G$.
\end{lemma}
\begin{proof}
  For each group~$i\in [c]$, we have $\sigma(\mathcal{U}_i)<\alpha n^4$ by
  assumption, and thus at most $\sqrt{\alpha}n^2$ pairs satisfy
  $(\load(u,v,\cU_i)-\mu(\cU_i))^2>\sqrt{\alpha}n^2$ and are thus
  $i$-atypical. As $c\sqrt{\alpha}< \beta/16$, the second assertion follows.
\end{proof}

For showing the quasirandomness of~$\tilde G$ we shall use the following easy error bound.
\begin{lemma}
\label{lem:HONZAerror}
For each $M\in (0,1]$ and each $a\in(-0.5,\infty)$, we have $M-|a|\le M^{1+a}\le M+|a|$.
\end{lemma}
\begin{proof}
Suppose that $M$ is fixed. The claim holds trivially for $a=0$. Thus it suffices to prove that within the range of $a$, the derivative of $M^{1+a}$ with respect to~$a$ is at most~1 in absolute value. We have $|\frac{\mathrm d}{\mathrm d a} M^{1+a}|=|M^{1+a}\ln M|\le|\sqrt{M}\ln M|$. It can be numerically checked, that for each $M\in (0,1]$, we have $\sqrt{M}\ln M\in(-0.8,0]$. The claim follows.
\end{proof}
\begin{lemma}
  \label{lem:quasirandom}
  With probability at least $1-\exp(-\sqrt{n})$, we
  have that $\tilde G$ is
  $\beta$-quasirandom.
\end{lemma}
\begin{proof}
  By the definition of quasirandomness (Definition~\ref{def:subsetQuasi})
  we need to show that with high probability there exists a number
  $p_{\tilde G}$ such that for each set $B\subseteq V$, we have
  that \[e(\tilde G[B])=p_{\tilde G}\binom{|B|}{2}\pm \beta n^2\;.\] As $G$
  is $\alpha$-quasirandom, it is enough to show that with high probability
  there is a number $p_h$ such that each set $B\subseteq V$ satisfies
  \[\Big|E(h) \cap \binom{B}{2}\Big|=p_h\binom{|B|}{2}\pm
  \frac{\beta n^2}2\;.\] Let us fix a set $B\subset V$. We first show that with
  high probability $|E(h)\cap \binom{B}{2}|$
  is close to its expectation
  $\lambda_B=\Exp[|E(h)\cap
      \binom{B}{2}|]$. Note that the
  random variable $|E(h)\cap\binom{B}{2}|$ is
  $\Delta^2$-Lipschitz. McDiarmid's Inequality, Lemma~\ref{lem:McDiarmid}, gives that
  \begin{align*}
    \Prob\bigg[\bigg|\Big|E(h)\cap
    \binom{B}{2}\Big|-\lambda_B\bigg|\ge \frac {\beta n^2}{8}\bigg]
		&\le 2\exp\left(-\frac{2\beta^2n^4}{64\Delta^4}  \cdot \frac {r}{4n^2}\right)
		\le \exp(-n^{19/10})\;.
  \end{align*}
  Since there are $2^m\le 4^n$ choices of the set $B$, the lemma will
  follow from a union bound, if we show that for each set $B$ we
  have
  \begin{equation}\label{eq:odhadph}
    \Exp\Big[\Big|E(h)\cap
        \binom{B}{2}\Big|\Big]=p_h\binom{|B|}{2}\pm\frac{\beta n^2}4\;.
  \end{equation}

  By Lemmas~\ref{lem:Nimportant} and~\ref{lem:atypical}, the total
  contribution to the number of edges in $E(h)$ from
  non-important groups and from atypical edges is at most $\beta n^2/8$.
  Thus~\eqref{eq:odhadph} follows if for each typical edge $uv$ of $G$
  the probability that there is an edge of a forest in an important group
  that gets mapped to $uv$ is $p_h\pm \alphaD$. 
  We shall prove that this is the case in Claim~\ref{cl:John}, which will
  conclude the proof of the lemma.

  Before turning to this claim, we consider a fixed important group~$i$ and
  bound the probability that a typical edge $uv$ is the image of any
  edge of a forest of this group.  Observe that it suffices to consider
  forests $F_{i,s}$ with $U_{i,s}\cap \{u,v\}=\emptyset$. Let
  $xy\in E(F_{i,s})$ for such a forest. Denote by $A(x,y,u,v)$ the event that $h(x)=u$ and
  $h(y)=v$. 
	Then by Lemma~\ref{lem:probh(x)=v}\ref{it:probedge} we have 
  \begin{equation}\label{eq:ProbAxy}
    \Prob\left[A(x,y,u,v)\right]=\big(1\pm \alphaA(\frac{1}{d})^\Delta\big)^{\Delta+2}\frac
  {1}{dm_{i,s}^2}\eqByRef{eq:Vis_samesize}(1\pm \alphaB)\frac {1}{dm_i^2}\;.
  \end{equation}

	Let $H^i_{uv}$ be the set of all ordered pairs $(x,y)$ such that $xy\in
  E(F_{i,s})$ for~$s$ with $U_{i,s}\cap\{u,v\}=\emptyset$ and
  \begin{equation}\label{eq:defMi}
    M_i(u,v)= \prod_{(x,y)\in H^i_{uv}}\Prob\left[\overline{{A(x,y,u,v)}}\right]\;.
  \end{equation}
  Note that $M_i(u,v)$ is the
  probability that~$uv$ is not used by any forest from
  group~$i$ in an alternative random experiment where the forest edges 
  are mapped to~$G$ independently. Our next goal is to show that in our random
  experiment the corresponding probability does not deviate much from $M_i(u,v)$.

  \begin{claim}\label{cl:Rheneas}
    For each $uv\in E(G)$ and each important group~$i$ we have
    \[\Prob[h^{-1}(uv)\cap \bigcup_{s\in [k_i]}E(F_{i,s})= \emptyset]=(1\pm
    \alpha)M_i(u,v)\,.\]
  \end{claim}
  \begin{claimproof}[Proof of Claim~\ref{cl:Rheneas}] 
    We want to use Suen's inequality. Let $uv\in E$ be fixed and abbreviate $A(x,y)=A(x,y,u,v)$.  We set up a
    superdependency graph for the events
    $\{A(x,y)\}_{(x,y)\in H^i_{uv}}$ as follows.
    For $(x,y),(x',y')\in H^i_{uv}$, define $(x,y)\sim (x',y')$ if
    $\dist(xy,x'y')\le 4$.
    Notice that the embedding of a primary vertex influences only
    the embedding of the vertices in its neighbourhood (and itself). The
    embedding of a secondary vertex on the other hand is independent of the
    embedding of all vertices of distance at least~$3$. As a consequence, we get
    that~$\sim$ indeed defines a superdependency graph for the events $A(x,y)$.  The degrees in
    the superdependency graph are at most $1+4\Delta^5\le 5\Delta^5$.
    For $(x,y),(x',y')\in H^i_{uv}$, set
    \begin{equation}\label{eq:Dm}
      \nu_{xy,x'y'}=\frac{\Prob[A(x,y)\cap A(x',y')]+ \Prob[A(x,y)]\cdot \Prob[A(x',y')]}{\prod(1-\Prob[A(\tilde x,\tilde y)])}\;,
    \end{equation}
    where the product in the denominator ranges through all $(\tilde
    x,\tilde y)\in H^i_{uv}$ such that $(x,y)\sim(\tilde x,\tilde y)$ or
    $(x',y')\sim(\tilde x,\tilde y)$. We next upper-bound~\eqref{eq:Dm} in the
    case that $(x,y)\neq (x',y')$ are such that $(x,y)\sim(x',y')$. The
    denominator in~\eqref{eq:Dm} has at most $10\Delta^5$ factors, each of
    which is at least $1-\frac{1+\alphaB}{dm_i^2}$
    by~\eqref{eq:ProbAxy}. Similarly, by~\eqref{eq:ProbAxy} 
    the terms $\Prob[A(x,y)]$ and $\Prob[A(x',y')]$ 
    are at most $\frac{1+\alphaB}{dm_i^2}$.

    The event $A(x,y)\cap A(x',y')$ is empty when $x'=y$, or $x=y'$. If
    $x=x'\in\secondary$ or $y=y'\in\secondary$, the event $A(x,y)\cap
    A(x',y')$ puts requirements
    on the placement of two primary vertices and one secondary vertex
    $t\in\{x',y'\}$. Analogously to the proof of
    Lemma~\ref{lem:probh(x)=v}\ref{it:probedge}, we can show that in this
    case this event has probability
    \begin{align*}
      \Prob[A(x,y)\cap A(x',y')]&=\frac{\big(1\pm\alphaA(\frac{2}{d})^\Delta\big)^{\Delta+2}}{dm_{i,s}^3}
      \eqByRef{eq:Vis_samesize}
      \frac{1\pm \alphaB}{dm_i^3}\;.
    \end{align*}

    It remains to consider the case when $\{x',y'\}\setminus \{x,y\}$
    contains a secondary vertex. Without loss of generality assume
    that~$y'$ is secondary. We first expose the limping homomorphism
    entirely, except for~$y'$. Two cases may occur: either~$y'$ is skipped,
    and therefore, $A(x',y')$ cannot occur, or the image of~$y'$ is
    selected uniformly among at least $d^\Delta m_i/2$ vertices.
    Using~\eqref{eq:ProbAxy} we have
    \begin{align*}\Prob[A(x',y')\cap A(x,y)]&=\Prob[A(x',y')|A(x,y)]\cdot \Prob[A(x,y)]
      \le \Prob[h(y')=v|A(x,y)]\cdot \Prob[A(x,y)]\\
      &\le \frac {2}{d^\Delta m_i}\cdot \frac {1+\alphaB}{dm_i^2}\le \frac {3}{d^{\Delta+1}m_i^3}\;.
    \end{align*}
    Thus, for all $(x,y)\neq(x',y')$ with $(x,y)\sim(x',y')$ we have
    \begin{equation}\label{eq:boundOnNu}
      \nu_{xy,x'y'}\le
      \frac{4}{d^{\Delta+1}m_i^3}\cdot 
      \frac{1}{\big(1-\frac{1+\alphaB}{dm_i^2}\big)^{10\Delta^5}}     \le
      \frac{5}{d^{\Delta +1}(\epsilon n)^3}\;.
    \end{equation}
    Suen's inequality (Lemma~\ref{lem:Suen}) states that
    \begin{equation}\begin{split}
     \bigg|\Prob\Big[h^{-1}(uv)\cap \bigcup _{s\in[k_i]}E(F_{i,s})=\emptyset\Big]-M_i(u,v)\bigg|
      =\bigg|\Prob\Big[\bigwedge_{(x,y)\in H^i_{uv}} \overline{A(x,y)}\Big]-M_i(u,v)\bigg|\\
      \le M_i(u,v)\bigg(
        \exp\Big(\sum_{(x,y)\sim (x',y')}\nu_{xy,x'y'}\Big)-1\bigg)\;.
      \label{eq:Suen1}
    \end{split}\end{equation}
  We use~\eqref{eq:boundOnNu}, the bound $5\Delta^5$ on the degrees in the
  superdependency graph, and the fact that we have at most $4n^2/r$ edges
  in $\bigcup_{s}E(F_{i,s})$ to obtain that
  \begin{align*}
    \sum_{xy\sim x'y'}\nu_{xy,x'y'}\le \frac{5}{d^{\Delta+1}\varepsilon^3n^3}\cdot
    \frac{4n^2}{r}\cdot 5\Delta^5= \frac{100\Delta^5}{d^{\Delta+1}\varepsilon^3rn}\;.
  \end{align*}
  In particular, as $n\ge n_0$ is large, we get $\sum_{xy\sim
    x'y'}\nu_{xy,x'y'}<\frac{\alpha}{2}<1$. We use that $\exp(a)-1\le 2a$
  for each $a\in (0,1)$ and get
  $\Prob[h^{-1}(uv)\cap \bigcup_{s\in[k_i]}E(F_{i,s})= \emptyset]=(1\pm \alpha
  )M_i(u,v)$.
\end{claimproof}

\begin{claim}\label{cl:John} There exists $p_h>0$ such that for each
  typical edge $uv\in E$ we have
  \[\Prob[h^{-1}(uv)\cap \bigcup_{i\in \impgr}\;\bigcup_{s\in [k_i]}E(F_{i,s})=\emptyset]=p_h\pm\alphaD\;.\]
\end{claim}
\begin{claimproof}[Proof of Claim~\ref{cl:John}]
  First fix $i\in\impgr$ and a typical edge $uv\in E$.
  Let $S=\big\{s\in[k_i]\colon U_{i,s}\cap\{u,v\}=\emptyset\big\}$.
  Observe that $|S|=k_i-\load(u,v,\cU_i)$ and \[|H^i_{u,v}|=\sum_{s\in
    S}2(n_{i,s}-1)=|S|2(1\pm\alpha)(n_i)-1=2\big(k_i-\load(u,v,\cU_i)\big)(n_i-1)\pm
  3n\alpha n_i\,.\] Let us write
  $\ell_{uv}=2\big(k_i-\load(u,v,\cU_i)\big)(n_i-1)$. Further, we write
  $\ell_i=2(k_i-\mu(\mathcal {U}_i)) (n_i-1)$, and
  $M_i=(1-\frac{1}{dm_i^2})^{\ell_i}$. Note that $\epsilon^2n^2\le
  \ell_{uv}\le 2nn_i$ because $k_i\ge\sqrt{\alpha}nr/2$ as $i\in\impgr$.
  Plugging~\eqref{eq:ProbAxy} into \eqref{eq:defMi}, we get
  \begin{align}
  \begin{split}\label{eq:clean1}
    M_i(u,v)&=\Big(1-\frac{1\pm \alphaB}{dm_i^2}\Big)^{|H^i_{uv}|}=\exp\left((\ell_{uv}\pm
      3\alpha nn_i)\cdot \ln\Big(1-\frac{1\pm \alphaB}{dm_i^2}\Big)\right)\\
    &=\exp\left((\ell_{uv}\pm
      3\alpha nn_i)\cdot (1\pm 2\alphaB)\cdot \ln\Big(1-\frac{1}{dm_i^2}\Big)\right)\\
    &=\exp\left((\ell_i\pm
      \alphaC nn_i)\cdot \ln\Big(1-\frac{1}{dm_i^2}\Big)\right)=\exp\left(\ell_i(1\pm
      \alphaD)\cdot \ln\Big(1-\frac{1}{dm_i^2}\Big)\right)\\
    &=\Big(1-\frac{1}{dm_i^2}\Big)^{(1\pm\alphaD)\ell_i}=M_i^{1\pm\alphaD}\;,
    \end{split}
\end{align}
where the third equality uses that
$$\ln\big(1-(1\pm\alphaB)\lambda)\big)=-(1\pm 1.5 \alphaB)\lambda=(1\pm 2\alphaB)\ln(1-\lambda)\quad\text{ for $|\lambda|\ll \alphaB$}\;,$$
and the fourth
equality uses that $uv$ is typical.
In total we get that for each typical edge $uv\in E$ we have
\begin{align}\label{eq:clean2}
\Prob\bigg[h^{-1}(uv)\cap \bigcup_{i\in \impgr}\;\bigcup_{s\in [k_i]}E(F_{i,s})=\emptyset\bigg]&= \prod_{i\in\impgr} M_i^{1\pm\alphaD}=\prod_{i\in\impgr}M_i\pm\alphaD\;,
\end{align}
where we used Lemma~\ref{lem:HONZAerror}.
The claim follows by setting $p_h=\prod_{i\in\impgr}M_i$.
\end{claimproof}
This finishes the proof of Lemma~\ref{lem:quasirandom}.
\end{proof}


Recall that $\tilde\cU_i=(\tilde U_{i,s})_{s\in [k_i]}$ with $\tilde U_{i,s}=U_{i,s}\cup V(h_{i,s})$.

 \begin{lemma}\label{lem:load}
We have
\[
\Prob[\forall i\in [c]\colon \sigma(\tilde\cU_i)\le \beta n^4]\ge
 1-\exp(-\sqrt{n}).
\]
 \end{lemma}
\begin{proof}
Fix $i\in[c]$. 
Let $L_i^*(u,v)=\load(u,v,\tilde\cU_i)-\load(u,v,\cU_i)$ and
$\mu_i^*=\mu(\tilde\cU_i)-\mu(\cU_i)$.
\begin{claim}\label{cl:gamman2}
With probability at least $1-\frac1c\exp(-\sqrt{n})$ we have that 
\begin{equation}\label{eq:GAMMA}
  \sum_{uv\in\binom{V}{2}}(L_i^*(u,v)-\mu_i^*)^2\le \alphaE n^4\;.
\end{equation}
\end{claim}
\begin{claimproof}[Proof of Claim~\ref{cl:gamman2}]
Fix an arbitrary $i$-typical pair $uv\in \binom{V}{2}$. Let~$s$
be such that $U_{i,s}\cap \{u,v\}=\emptyset$ and let $x\in
V(F_{i,s})$ be  arbitrary. Denote by~$A_x$ the event that 
$h(x)\in  \{u,v\}$. Lemma~\ref{lem:probh(x)=v}\ref{it:prob-prim}
and~\ref{it:prob-sec} and~\eqref{eq:Vis_samesize} give that 
\begin{align}\label{eq:Ax}
\Prob[A_x]=\frac {2(1\pm
\alphaB)}{m_i}\le \frac {3}{m_i}\;.
\end{align} 
Set $M=\prod_{x\in V(F_{i,s})}(1-\Prob[A_x])$. We have $M=\left(1-\frac {2\pm
2\alphaB}{m_i}\right)^{n_{i,s}}$.
Recall that $n_{i,s}=(1\pm\alpha)n_i$. We can now manipulate the error bounds as in~\eqref{eq:clean1},~\eqref{eq:clean2} and get
\begin{align}\label{eq:loadM}
M=\left(1-\frac{2}{m_i}\right)^{n_i}\pm\alphaC\;.
\end{align}

We shall approximate the values $p_{i,s}(u,v)=\Prob[h^{-1}(\{u,v\})\cap
V(F_{i,s})=\emptyset]$ using Suen's Inequality, similarly as in the proof
of Claim~\ref{cl:Rheneas}. We define a superdependency graph on vertex
set $V(F_{i,s})$ for the events $\{A_x\}_{x\in V(F_{i,s})}$ 
by letting $x\sim y$ whenever $\dist(x,y)\le 3$. Notice that the superdependency graph has degree at most~$\Delta^3$. Let
\begin{align}\label{eq:lete}
\nu_{xy}=\frac{\Prob[A_x\cap A_y]+\Prob[A_x]\cdot\Prob[A_y]}{\prod(1-\Prob[A_z])}\;,
\end{align}
where the product in the denominator is over all $z$ with $z\sim x$ or $z\sim
y$. By Lemma~\ref{lem:AxAy} we have
$
  \Prob[A_x\cap A_y]
  \le 4\big(\frac{3}{d}\big)^{4\Delta^2}\frac{1}{m_{i,s}^2}
  \le 5\big(\frac{3}{d}\big)^{4\Delta^2}\frac{1}{m_{i}^2}
$. Notice that the product in the denominator has at
most $2\Delta^3$ factors, corresponding to the size of the union of the neighbourhoods of~$x$ and~$y$. 
Together with~\eqref{eq:Ax},  we get for each $x\sim y$ that
\begin{equation*}
  \nu_{xy}
  \le \frac{5(\frac{3}{d})^{4\Delta^2}\frac{1}{m_{i}^2}+\frac{9}{m_i^2}}
           {(1-\frac{3}{m_i})^{2\Delta^3}}
  \le 10\Big(\frac{3}{d}\Big)^{4\Delta^2}\frac{1}{m_i^2}\,.
\end{equation*}
Note that for each $s\in [k_i]$, there are at most $(\Delta+1)^3m_{i,s}$
pairs $x,y\in V(F_{i,s})$ with $x\sim y$. Hence Suen's Inequality
(Lemma~\ref{lem:Suen}) gives
\begin{equation}\begin{split}\label{eq:boundpisuv}
    p_{i,s}(u,v)&=\Prob\Big[\bigwedge_{x\in V(F_{i,s})} \overline{A_x}\Big]=M \pm  M\cdot
    \left(\exp\bigg(\frac{(\Delta+1)^311(\frac{3}{d})^{4\Delta^2}}{m_i}\bigg)-1\right) \\ 
    &\leByRef{eq:loadM}\left(1-\frac{2}{m_i}\right)^{n_i}\pm 2\alphaC\;.
\end{split}\end{equation}

Set $p_i= \left(1-\frac{2}{m_i}\right)^{n_i}$.  We will show that
\begin{equation}\label{eq:domination}
\Prob[L_i^*(u,v)\ge (k_i-\mu(\mathcal {U}_i))\cdot p_i+4\alphaC n]\le \exp(-\alphaC^2n)\;.
\end{equation}
The random variable~$L_i^*(u,v)$ has law $\sum_{s\::\: U_{i,s}\cap
\{u,v\}=\emptyset}\mathrm{Be}(p_{i,s}(u,v))$. Using~\eqref{eq:boundpisuv}, we get
that $L_i^*(u,v)$ is stochastically dominated by $\sum \mathrm{Be}(p_i+2\alphaC)$, where the sum runs through
all~$s$ such that $U_{i,s}\cap \{u,v\}=\emptyset$. The number of summands is
$k_i-\load(u,v,\cU_i)$, which is at most $k_i-\mu(\mathcal {U}_i)+\sqrt[4]{\alpha}n$,
as $uv$ is $i$-typical. Observe that $(k_i-\mu(\mathcal {U}_i)+\sqrt[4]{\alpha}n)(p_i+2\alphaC)\le
(k_i-\mu(\mathcal {U}_i))\cdot p_i+3\alphaC n$. By Chernoff's
inequality~\eqref{eq:Ch1} and because $k_i\le 2n$, we
obtain~\eqref{eq:domination}. The computation that
$\Prob[L_i^*(u,v)\le (k_i-\mu(\mathcal {U}_i))\cdot p_i-4\alphaC n]\le \exp(-\alphaC^2n)$
is done analogously. So with probability at least $1-2\binom{m}{2}\cdot
\exp(-\alphaC^2n)\ge 1-\frac1c\exp(-\sqrt{n})$, all $i$-typical pairs~$uv$ satisfy
$L_i^*(u,v)=(k_i-\mu(\mathcal {U}_i))p_i\pm 4\alphaC n$. Suppose this is
the case. Then 
\begin{align*}
\mu_i^*&=\frac{1}{\binom{m}{2}}\sum_{uv\in \binom{V}{2}}
=\frac{1}{\binom{m}{2}}\sum_{uv \text{ $i$-typical}}\left((k_i-\mu(\mathcal {U}_i))p_i\pm
4\alphaC n\right)
+ \frac{1}{\binom{m}{2}}\sum_{uv \text{ $i$-atypical}} L_i^*(u,v) \\
&=(k_i-\mu(\mathcal {U}_i))p_i\pm 4\alphaC n \pm
\frac{\sqrt[4]{\alpha}n^2\cdot 2n}{\binom{m}{2}}=(k_i-\mu(\mathcal {U}_i))p_i\pm 5\alphaC n \,,
\end{align*}
where we used Lemma~\ref{lem:atypical}.
So with probability   at
least $1-\frac1c\exp(-\sqrt{n})$ we have
\begin{align*}
\sum_{uv\in \binom{V}{2}}(L_i^*(u,v)-\mu_i^*)^2&\le \sum_{uv \text{
$i$-typical}}(9\alphaC n)^2+ \sum_{uv \text{ $i$-atypical}}n^2\\ &\le n^2 \cdot
100\alphaC^2n^2+\sqrt[4]{\alpha} n^2\cdot n^2\le \alphaE n^4\;,
\end{align*}
where we used Lemma~\ref{lem:atypical} again.
\end{claimproof}

\begin{claim}\label{cl:variance}
If~\eqref{eq:GAMMA} holds, then $ \sigma(\tilde\cU_i)\le \beta n^4$.
\end{claim}
\begin{claimproof}[Proof of Claim~\ref{cl:variance}]
  Let~$W_i$ be the set of all pairs $uv\in\binom{V}{2}$ such that $uv$ is
  $i$-atypical or $(L_i^*(u,v)-\mu_i^*)^2>\sqrt{\alphaE}n^2$. From
  Lemma~\ref{lem:atypical} and~\eqref{eq:GAMMA} we get that 
  $|W_i|\le\sqrt[4]{\alpha}n^2+\sqrt{\alphaE}n^2< 2\sqrt[4]{\alphaE}n^2$. So,
  \begin{align*}
    \sigma(\tilde\cU_i)&=\sum_{uv\in\binom{V}{2}}(\load(u,v,\cU_i)+L_i^*(u,v)-\mu(\mathcal{U}_i)-\mu_i^*)^2\\
    &=\sum_{uv\in\binom{V}{2}}(\load(u,v,\mathcal{U}_i)-\mu(\mathcal{U}_i))^2+\sum_{uv\in\binom{V}{2}}(L_i^*(u,v)-\mu_i^*)^2\\
    &\qquad\qquad+\sum_{uv\in\binom{V}{2}}2(\load(u,v,\mathcal{U}_i-\mu(\mathcal{U}_i))(L_i^*(u,v)-\mu_i^*)\\
    &\leByRef{eq:GAMMA} \alpha n^4+\alphaE n^4+2\bigg(\sum_{uv\in W_i}n^2+\sum_{uv\in \binom{V}{2}\setminus W_i}(\sqrt[4]{\alpha}n\cdot
    \sqrt[4]{\alphaE}n)\bigg)\\ 
    &\le \alpha n^4+\alphaE n^4+2(2\sqrt[4]{\alphaE}n^2\cdot n^2+n^2\cdot \sqrt{\alphaE}n^2)<\beta n^4\;.
  \end{align*}
\end{claimproof}
Claims~\ref{cl:gamman2} and~\ref{cl:variance} and a union bound over all
$i\in[c]$ imply Lemma~\ref{lem:load}.
\end{proof}

\begin{lemma}\label{lem:homogeneous}
  With probability at least $1-\exp(-\sqrt{n})$ we have for each $i\in[c]$
  that $\big||\tilde U_{i,s}|-|\tilde U_{i,s'}|\Big|\le \beta n$, for all
  $s,s'\in [k_i]$.
\end{lemma}

\begin{proof}
  We first compute the expected size of the image $V(h_{i,s})$. More than
  the exact value of the expected size, we need to show that it does not
  depend much on $s\in [k_i]$. This is done in~\eqref{eq:EsizeImage}. Then
  we show the concentration.

  Fix $(i,s)$ and fix $v\in V(G_{i,s})$.  For $x\in V(F_{i,s})$ denote by $A_x$ the event
  that~$x$ is mapped to~$v$. By
  Lemma~\ref{lem:probh(x)=v}\ref{it:prob-prim} and~\ref{it:prob-sec} we
  have that
  \begin{align}\label{eq:IS1}
    \Prob[A_x]=\frac{(1\pm \alphaA)^{\Delta+3}}{m_{i,s}}
    \eqByRef{eq:Vis_samesize}
    \frac{1\pm\alphaB}{m_i}\;.
  \end{align}
  Using Suen's Inequality, we shall approximate 
	$\Prob[h_{i,s}^{-1}(v)\cap V(F_{i,s})=\emptyset]=\Prob\big[\bigwedge_{x\in V(F_{i,s})}\overline{A_x}\big]$ by
  $M=\prod_{x\in V(F_{i,s})}\Prob[\overline{A_x}]$.
%
	%
	%
	%
	Manipulating the error bounds same as 
	in~\eqref{eq:loadM}, we have that
  \begin{align}\label{eq:M=}
    M=\left(1-\frac{1\pm \alphaB}{m_i}\right)^{n_{i,s}}=\left(1-\frac{1}{m_i}\right)^{n_i}\pm \alphaC\;.
  \end{align}
For $x,y\in V(F_{i,s})$, we write $x\sim y$ if
  $\dist(x,y)\le 2$. Note that this defines a superdependency graph
  for the events $\{A_x\}_{x\in V(F_{i,s})}$. Let
  \begin{align*}
    \nu_{xy}=\frac{\Prob[A_x\cap A_y]+\Prob[A_x]\cdot \Prob[A_y]}{\prod (1-\Prob[A_z])},
  \end{align*}
  where the product in the denominator is over all~$z$ with $z\sim x$ or
  $z\sim y$.  
	The product in the denominator has at most
  $2(\Delta^2+1)$ terms. We infer from~\eqref{eq:IS1} that the denominator
  is at least $1/2$ and that $\Prob[A_x]\cdot \Prob[A_y]$ is at most
  $(1+3\alphaB)/m_{i}^2$.
	Lemma~\ref{lem:AxAy} and~\eqref{eq:Vis_samesize}
  give $\Prob[A_x\cap A_y]\le
  (\frac{3}{d})^{4\Delta^2}\frac{1}{m_{i,s}^2}\le
  (\frac{4}{d})^{4\Delta^2}\frac{1}{m_{i}^2}$. Thus, we get that
  $\nu_{xy}\le (\frac{5}{d})^{4\Delta^2}\frac{1}{m_{i}^2}$.  Suen's
  Inequality (Lemma~\ref{lem:Suen}) gives that
  \begin{multline*}
    \Prob[h_{i,s}^{-1}(v)\cap V(F_{i,s})=\emptyset]
    =M\bigg(1\pm \Big(\exp(\Delta^2n_{i,s}\cdot \Big(\frac{5}{d}\Big)^{4\Delta^2}\frac{1}{m_{i}^2}\Big)-1\big)\bigg)\\
    \eqByRef{eq:M=} \bigg(\Big(1-\frac{1}{m_i}\Big)^{n_i}\pm \alphaC\bigg)
      \bigg(1\pm
      \Big(\exp\Big(\Delta^2\Big(\frac{5}{d}\Big)^{4\Delta^2}\cdot\frac{8}{r\eps^2
      n}\Big)-1\Big)\bigg)
    =\Big(1-\frac{1}{m_i}\Big)^{n_i}\pm 2\alphaC\;.
  \end{multline*}
  Therefore the expected size of the image $V(h_{i,s})$ is
  \begin{equation}\begin{split}
      \Exp[|V(h_{i,s})|]&=\sum_{v\in V(G_{i,s})}\Prob[h_{i,s}^{-1}(v)\cap V(F_{i,s})\neq\emptyset]
      =m_{i,s}\cdot \left(1-\Big(1-\frac{1}{m_i}\Big)^{n_i}\pm
        2\alphaC\right) \\
      &\eqByRef{eq:Vis_samesize}m_i \left(1-\Big(1-\frac{1}{m_i}\Big)^{n_i}\pm 3\alphaC\right)\;.
      \label{eq:EsizeImage}
  \end{split}\end{equation}
  Now we use McDiarmid's Inequality to show the concentration of $|V(h_{i,s})|$.
%
	%
	Note that $|V(h_{i,s})|$ is $(\Delta+1)$-Lipschitz.
  Hence by McDiarmid's Inequality, Lemma~\ref{lem:McDiarmid}, we have
  \begin{align*}
    \Prob[\big|\Exp[|V(h_{i,s})|]-|V(h_{i,s})|\big| >\beta n/4]
    \le 2\exp\Big(-\frac{2\beta^2n^2}{16(\Delta+1)^2n_{i,s}}\Big)
    =\exp\Big(-n^{2/3}\Big)\;.
  \end{align*}
  Set $H_i=m_i \left(1-\Big(1-\frac{1}{m_i}\Big)^{n_i}\right)$. Then
  $\Exp[|V(h_{i,s})|]=H_i\pm 6\alphaC n$.  As
  $\big||U_{i,s}|-|U_{i,s'}|\big|\le \alpha n$, by a union bound over all
  $s\in [k_i]$ we obtain that
  \begin{align*}
    \Prob[\exists 
    s,s'\in [k_i]\::\:
    \big||\tilde U_{i,s}|-|\tilde U_{i,s'}|\big|
    >\beta n]\le \tfrac1c\exp(-\sqrt{n})\;.
  \end{align*}
  A union bound over all $i\in [c]$ leads to the statement of the lemma.
\end{proof}

\section{Concluding remarks}
\label{sec:Concluding}
In this section we discuss various ways how our main result, Theorem~\ref{thm:tree-packing}, could be extended.
\subsection{Strengthening Theorem~\ref{thm:tree-packing}: approximation}\label{ssec:strengthenapprox}
Theorem~\ref{thm:tree-packing} does not hold for $\epsilon=0$. To see this,
fix $\Delta\ge 3$ odd, let $\ell\ge 2$ be arbitrarily large, and consider
the full $\Delta$-regular tree of depth $\ell$ as in Figure~\ref{fig:3reg},
that is, each vertex in this tree has degree either~1 or~$\Delta$.  This
tree has an even number of leaves and an even number of internal vertices,
hence its order~$n$ is even.
\begin{figure}[h!]
     \centering
     \subfigure[The 3-regular tree of depth 2.]{\hspace{.5in}
          \label{fig:3reg}
          \includegraphics{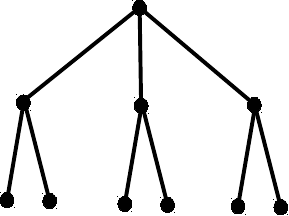}
          \hspace{.5in}}
     \hspace{0.2in}
     \subfigure[An example of the modified 3-regular tree.]{\hspace{.7in}
          \label{fig:3regm}
          \includegraphics{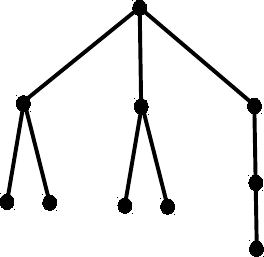}\hspace{.8in}}
     \caption{Regular trees and modified regular trees}
\end{figure}
Consider a family of $\frac n2$ copies of this tree. 
This family has $\binom{n}{2}$ edges in total. If it does not pack into $K_n$ we are done. 
Otherwise, in any such packing a vertex $v$ of $K_n$ accommodates exactly~$c_1$ leaves and~$c_2$ internal vertices of the trees, where $c_1$ and $c_2$ are integral and determined by the system
\begin{align*}
c_1+c_2&=\tfrac n2 &\text{(each tree uses $v$),}\\
c_1+\Delta c_2&=n-1 &\text{(each edge incident with $v$ is used).}
\end{align*}
This system has a unique solution (and thus the same for all vertices $v$) where~$c_1$ is half the number of leaves of one tree and~$c_2$ is half the number of internal vertices.

Now, we modify one of the trees by chopping off one leaf and appending it
to another leaf; see Figure~\ref{fig:3regm} (the resulting tree is
  not uniquely determined). This modified family does not pack into
$K_n$. Indeed, if it did then the vertex of $K_n$ hosting the unique vertex
of degree~2 would have to host $\tilde c_1$ leaves and $\tilde c_2$
vertices of degree $\Delta$, with
\begin{align*}
1+\tilde c_1+\tilde c_2&=\tfrac n2\;,&\text{$\hspace*{7cm}$}\\
2+\tilde c_1+\Delta \tilde c_2&=n-1\;.
\end{align*}
The integrality of the solution of the original system implies that the current one is not integral, contradiction. 

On the other hand, the following strengthening of Theorem~\ref{thm:tree-packing} may be true: Any family of trees of orders at most $n$ and maximum degrees at most $\Delta$ whose total number of edges is at most $\binom{n}{2}$ packs into $K_{n+C_\Delta}$, for a suitable constant $C_\Delta$ depending on $\Delta$ only.

\subsection{Strengthening Theorem~\ref{thm:tree-packing}: maximum degree}
We are convinced that at an expense of a more involved analysis, our techniques would allow to prove a version of Theorem~\ref{thm:tree-packing} (for each fixed $\epsilon>0$) for $\Delta$ growing with $n$, possibly as big as $\Delta=O(\log^\alpha n)$ for some $\alpha>0$.

We believe that Theorem~\ref{thm:tree-packing} holds even for $\Delta=\frac n2$. (New techniques would be necessary for a proof.) The following example shows that the $\frac n2$ barrier can essentially not be exceeded. Suppose that $\epsilon\in(0,10^{-3})$ is fixed. Let us consider a family of $\ell=\left\lfloor\binom{n}{2}/((\frac12+2\sqrt\epsilon)n)\right\rfloor$ 
copies of the star of order $(\frac12+2\sqrt\epsilon)n+1$. Note that $\ell<(1-3\sqrt\epsilon)n$. The total number of edges in this family is between $\binom{n}{2}-n$ and $\binom{n}{2}$. We claim it does not pack into $K_{(1+\epsilon)n}$. Suppose it does, and let us fix a packing. Let $W\subset V(K_{(1+\epsilon)n})$ be the vertices that do not host the centres of the stars. Observe that $|W|>3\sqrt\epsilon n$. Observe also that no edge of the packing lies inside $W$. That means that all the edges of the stars must be accommodated in the set $E(K_{(1+\epsilon)n})\setminus \binom{W}{2}$. We have
$$\binom{n}{2}-n>\binom{(1+\epsilon)n}{2}-\binom{3\sqrt\epsilon n}{2}\;,$$
a contradiction.
\medskip

Note that if the orders of the trees are at most half of the order of the host graph, no example analogous to that in Section~\ref{ssec:strengthenapprox} can be found.
Moreover, in Ringel's Conjecture (Conjecture~\ref{conj:Ringel}), it follows from the assumption on the order of the tree that its
maximum degree is at most half of the order of the host graph. 
Thus we propose the following strengthening of Conjecture~\ref{conj:Ringel}. 
\begin{conjecture}
Any family of trees of individual orders at most $n+1$ and total number of edges at most $\binom{2n+1}{2}$ packs into $K_{2n+1}$.
\end{conjecture}

\subsection{Different host graphs in Theorem~\ref{thm:tree-packing}}
Hobbs, Bourgeois, and Kasiraj~\cite{HBK87} modified
Conjecture~\ref{conj:tree-packing} to the setting of
complete bipartite graphs.
\begin{conjecture}\label{conj:tree-packingBip}
If $n$ is even then any family of $n$ trees $(T_j)_{j\in [n]}$ with $v(T_j)=j$ packs into $K_{n-1,n/2}$. 
If $n$ is odd then any family of $n$ trees $(T_j)_{j\in [n]}$ with $v(T_j)=j$ packs into $K_{n,(n-1)/2}$. 
\end{conjecture}
Our proof of Theorem~\ref{thm:tree-packing}
can be adjusted with only minor modifications to the bipartite setting. Thus, the very same
method yields an asymptotic solution of
Conjecture~\ref{conj:tree-packingBip} for trees of
bounded maximum degree. In that setting, the ratio of the host graph's colour classes does not have to be~$1:2$; one just needs them to be of the same order of magnitude.
\begin{theorem}\label{thm:tree-packingBip}
For any $\eps>0$ and any $\Delta\in \mathbb N$ there is an $n_0\in \mathbb N$
such that for any $a,b\ge n_0$, $\frac{a}{b}\in\left(\epsilon,\epsilon^{-1}\right)$ the following holds. Any
family of trees $(T_i)_{i\in[t]}$ with maximum degree
at most $\Delta$ and order at most $\min\{a,b\}$ satisfying
$\sum_{i=1}^t e(T_i)\le ab$ packs into $K_{(1+\eps)a,(1+\eps)b}$.
\end{theorem}

Also, it is clear that the proof of Theorem~\ref{thm:tree-packing} goes through when the graph $K_{(1+\epsilon)n}$ is replaced by an arbitrary dense quasirandom graph (and the condition on the total number of edges in the family of trees is adjusted accordingly). Packing in random and quasirandom graphs is an important direction of research for its own sake, see e.g.~\cite{BFKL}.

\subsection{The tree-packing process}
We expect that the random embedding process described in Section~\ref{sec:outline} performs well even as a dynamic process on an evolving graph. That is, we believe that the quasirandomness of the host graph is also maintained by a sequential random embedding of the trees, where we forbid the edges (globally) and vertices (just for that particular tree) immediately after they are used.
This would yield another proof of Theorem~\ref{thm:tree-packing}, but we believe the analysis of this process would also be interesting in its own right.

\section{Acknowledgement}
JH wishes to thank Demetres Christofides, G\'abor Kun and Oleg Pikhurko for helpful
discussions. We thank Codru\unichar{355} Grosu for pointing out an error in a previous version of the manuscript. Finally, we thank an anonymous referee for their very detailed comments on the manuscript.

Much of the work was done during research visits where we (JB, JH, DP) had to take our little children with us. We would like to acknowledge the support of the London Mathematical Society (JB, JH, DP), EPSRC Additional Sponsorship with grant reference EP/J501414/1 (DP), and the Mathematics Institute at the University of Warwick (JH) for contributing to childcare expenses that incurred during these trips.

The paper was finalised during the participation in the program 
\emph{Graphs, Hypergraphs, and Computing} at Institut Mittag--Leffler. We would like to thank the organisers and the staff of the institute for creating a very productive atmosphere. Moreover, we would like to thank Emili Simonovits for helping us with babysitting.

\bibliographystyle{amsplain} \bibliography{bibl}
\end{document}